\documentclass[reqno,11pt]{amsart}
\usepackage{amsmath,amsthm,amscd,amssymb,graphicx,enumerate,latexsym}

\usepackage[raiselinks,colorlinks]{hyperref}
\hypersetup{citecolor=blue}

\numberwithin{equation}{section}

\theoremstyle{plain}
\newtheorem{thm}{Theorem}
\newtheorem{lemma}{Lemma}[section]
\newtheorem{prop}[lemma]{Proposition}
\newtheorem{cor}[lemma]{Corollary}
\theoremstyle{definition}

\theoremstyle{remark}
\newtheorem{remark}{Remark}[section]

\def\Im{\mathop{\rm Im}\nolimits}

\newcommand{\s}{\text{\rm{s}}}

\DeclareMathOperator*{\supp}{supp}
\DeclareMathOperator*{\arccosh}{arccosh}
\allowdisplaybreaks

\title[Eventually monotone coefficients]{Spectral edge behavior for eventually monotone Jacobi and Verblunsky coefficients}
\author{Milivoje Lukic}
\address{Rice University, 6100 Main Street, Mathematics MS 136, Houston, TX 77005}
\date\today
\email{milivoje.lukic@rice.edu}
\thanks{The author was partially supported by NSF Grant DMS-1301582}


\begin{document}

\begin{abstract}
We consider Jacobi matrices with eventually increasing sequences of diagonal and off-diagonal Jacobi parameters. We describe the asymptotic behavior of the subordinate solution at the top of the essential spectrum, and the asymptotic behavior of the spectral density at the top of the essential spectrum.

In particular, allowing on both diagonal and off-diagonal Jacobi parameters perturbations of the free case of the form $- \sum_{j=1}^J c_j n^{-\tau_j} + o(n^{-\tau_1-1})$ with $0 < \tau_1 < \tau_2 < \dots < \tau_J$ and $c_1>0$, we find the asymptotic behavior of the $\log$ of spectral density to order $O(\log(2-x))$ as $x$ approaches $2$.

Apart from its intrinsic interest, the above results also allow us to describe the asymptotics of the spectral density for orthogonal polynomials on the unit circle with real-valued Verblunsky coefficients of the same form.
\end{abstract}

\maketitle

\section{Introduction}

Given a compactly supported nontrivial probability measure $\mu$ on $\mathbb{R}$ (we follow standard usage in using nontrivial to mean not supported on a finite set of points), orthonormal polynomials $p_n(x)$, $n=0,1,2,\dots$ are obtained by applying the Gram--Schmidt process to $1, x, x^2, \dots$ so that
\[
\int p_m(x) p_n(x) d\mu(x) = \delta_{m,n}.
\]
The polynomials obey the relation
\[
x p_n(x) = a_{n} p_{n-1}(x) + b_{n+1} p_n(x) + a_{n+1} p_{n+1}(x)
\]
for some bounded sequences of coefficients $a_n > 0$ and $b_n \in \mathbb{R}$, which gives the classical correspondence between the measure $\mu$ and its Jacobi coefficients $\{a_n, b_n\}_{n=1}^\infty$ \cite{Teschl,SimonRice}. Conversely, this correspondence can be realized by observing the half-line Jacobi matrix which acts on $\ell^2(\mathbb{N})$ by
\[
(Ju)_n = \begin{cases}
a_{n-1} u_{n-1} + b_n u_n + a_n u_{n+1}  & n \ge 2 \\
b_1 u_1 + a_1 u_2 & n=1
\end{cases}
\]
whose spectral measure with respect to $\delta_1$ is $\mu$.

In spectral theory, the free Jacobi matrix refers to the choice of coefficients $a_n \equiv 1$, $b_n \equiv 0$, which corresponds to the measure
\begin{equation}\label{freemeasure}
\frac 1{2\pi} \chi_{(-2,2)}(x) \sqrt{4-x^2}\,dx
\end{equation}
supported on $[-2,2]$. There is a vast literature on decaying perturbations of the free case, relating decay properties of the perturbation to spectral properties of the measure \cite{DenisovKiselevreview,Killipreview}. For instance, the main focus of this paper will be on Jacobi parameters such that
\begin{equation}\label{anbnincrease}
b_n \le b_{n+1}, \qquad a_n \le a_{n+1}, \qquad\text{for }n\ge N_0
\end{equation}
and
\begin{equation}\label{anbnlimit}
\lim_{n\to\infty} b_n = 0, \qquad \lim_{n\to\infty} a_n = 1.
\end{equation}
Such sequences $(a_n)_{n=1}^\infty$, $(b_n)_{n=1}^\infty$ have bounded variation so,  by Weidmann's theorem \cite{Weidmann67,MateNevai83}, the corresponding measures are of the form
\begin{equation}\label{measuremu}
d\mu = f(x) dx + d\mu_\s
\end{equation}
with $f(x)$ continuous and strictly positive on $(-2,2)$, $f=0$ on $\mathbb{R} \setminus (-2,2)$, and $d\mu_\s$ supported on $\mathbb{R} \setminus (-2,2)$.

While such general results are available that imply continuity and strict positivity of $f$ on $(-2,2)$, the asymptotic behavior of $f(x)$ as $x \to \pm 2$ is more sophisticated. Note that a change in asymptotics can be obtained by a compactly supported perturbation; e.g., compare the free case to the measure
\begin{equation}\label{chebyshev}
\chi_{(-2,2)}(x) \frac 1{\pi \sqrt{4-x^2}}\,dx
\end{equation}
which only differs from it in the value of the Jacobi coefficient $a_1 = \sqrt 2$ but has different asymptotic behavior as $x \to \pm 2$.

Beyond the intrinsic interest in the asymptotic behavior of the spectral density for decaying perturbations, this has emerged as a tool for proving higher-order Szeg\H o theorems of arbitrarily high order \cite{Lukic16}. A looser comparison can be drawn with results about density of states for ergodic operators, such as the interpretation for regular measures of the density of states as an equilibrium measure \cite{Simonequilibrium} (which, for sufficiently nice spectra, implies square-root behavior at spectral edges), or the phenomenon of Lifshitz tails for random operators \cite{Kirsch08}, which describes the rapid decay of the density of states at a spectral edge.

While our main interest is in the behavior of spectral density as $x \to 2$ (which corresponds to a double limit, $n\to\infty$ followed by $x\to 2$), the analysis also requires a discussion of the eigensolutions at $x=2$ as $n\to \infty$. Recall first that a (formal) eigensolution at $x$ is any sequence $u=(u_n)_{n=1}^\infty$ which solves for all $n\ge 2$ the recurrence relation
\begin{equation}\label{eigenx}
a_{n-1} u_{n-1} + b_n u_n + a_n u_{n+1} = x u_n.
\end{equation}
Following Gilbert--Pearson \cite{GilbertPearson87}, a nontrivial eigensolution $v = (v_n)_{n=1}^\infty$ is called subordinate if for any eigensolution $u$ which is not a multiple of $v$,
\[
\lim_{N \to \infty} \frac{ \sum_{n=1}^N \lvert v_n\rvert^2} {\sum_{n=1}^N \lvert u_n \rvert^2} =0.
\]
It is clear that if a subordinate eigensolution exists, subordinate eigensolutions (together with the trivial eigensolution) form a one-dimensional subspace of the two-dimensional space of eigensolutions.

Our first theorem concerns the existence of the subordinate solution at $x=2$ and its asymptotic behavior as $n \to \infty$.

\begin{thm}\label{thm1}
\begin{enumerate}[(a)]
\item If there exists $N_0$ such that $a_n \le 1$ and $b_n \le 0$ for all $n\ge N_0$, then there exists a subordinate solution $v = (v_n)_{n=1}^\infty$ at $x=2$, and $v \in \ell^\infty(\mathbb{N})$.

\item If, moreover, \eqref{anbnincrease} and \eqref{anbnlimit} hold, then there exist constants $C_1, C_2 \in (0,\infty)$ such that for all $n \ge N_0$,
\begin{equation}\label{decayingsolnasymptotics}
C_1 n^{-1} e^{-\sum_{j=N_0}^n \Gamma_j} \le v_n \le C_2 n e^{-\sum_{j=N_0}^n \Gamma_j},
\end{equation}
where $\Gamma_n$ is defined for all $n \ge N_0$ by
\[
\Gamma_n  = \arccosh \frac{2-b_n}{2a_n}.
\]
\end{enumerate}
\end{thm}

Our second theorem considers the asymptotic behavior of the spectral density $f(x)$ as $x \nearrow 2$. Note that $(p_{n-1}(x))_{n=1}^\infty$ is a solution of \eqref{eigenx}, sometimes described as the Dirichlet solution; our next theorem will assume that the Dirichlet solution at $x=2$ is not a subordinate solution. To motivate the relevance of this condition, note that in the two examples \eqref{freemeasure}, \eqref{chebyshev} considered above, eigensolutions at $x=2$ are for $n\ge 2$ linear sequences $u_n = A + B n$, and solutions with $B=0$, $A\neq 0$ are subordinate. An exact calculation for small $n$ shows that while for the free case \eqref{freemeasure} $p_{n-1}(2) = n$, for the measure \eqref{chebyshev} $p_{n-1}(2) = \sqrt{2}$ for $n\ge 2$, a subordinate solution.

In preparation for our next theorem, let us note that when \eqref{anbnincrease}, \eqref{anbnlimit} hold and $x \in [b_{N_0}+2a_{N_0},2)$, we can define 
\begin{equation}\label{Nx}
N(x) = \max \{ n \in \mathbb{N} \mid b_n + 2 a_n \le x \}
\end{equation}
As long as $b_{N_0}+2a_{N_0}$ is strictly smaller than $2$ (which holds if the sequences $a_n, b_n$ are not both eventually constant), the set is nonempty and bounded for $x \in [b_{N_0}+2a_{N_0},2)$ so $N(x)$ is an integer-valued function there. For $x$ in this interval and for  $N_0 \le n \le N(x)$, we also define
\begin{equation}\label{gammanxdef}
\gamma_n(x) = \arccosh \frac{x-b_n}{2a_n}.
\end{equation}
\begin{thm}\label{thm2}
Let $a_n>0$, $b_n \in \mathbb{R}$ be Jacobi parameters for the measure \eqref{measuremu} and assume that \eqref{anbnincrease}, \eqref{anbnlimit} hold and that the sequences $a_n, b_n$ are not both eventually constant. If $(p_{n-1}(2))_{n=1}^\infty$ is not a subordinate solution at $x=2$, then
\begin{equation}\label{eq18}
\left \lvert \log f(x) + 2 \sum_{n=N_0}^{N(x)} \gamma_n(x) \right\rvert \le 2 h(x) + O(1), \qquad x\nearrow 2
\end{equation}
where
\begin{equation}\label{eq19}
e^{h(x)} = N(x) (b_{N(x)+2} - b_{N(x)+1} + a_{N(x)+2} - a_{N(x)+1})^{-1} (2-x)^{1/2}.
\end{equation}
\end{thm}

Monotone Jacobi parameters were previously considered by Kreimer--Last--Simon \cite{KreimerLastSimon09}; our Theorem~\ref{thm2} generalizes their results in two ways. \cite{KreimerLastSimon09} considered separately two cases, $a_n \equiv 1$ and $b_n \equiv 0$, while our theorem allows both sequences to be non-constant, unifying and generalizing their arguments. \cite{KreimerLastSimon09} also assumed monotonicity from $N_0=1$, while we only assume monotonicity from some arbitrary $N_0$; while compactly supported perturbations are often easy to handle in spectral theory, in this problem having eventual monotonicity rather than monotonicity requires us to control the two-dimensional space of eigensolutions, instead of just the Dirichlet eigensolution, and introduces the complications related to subordinacy. 

\begin{remark}
\begin{enumerate}[(a)]
\item $(p_{n-1}(2))_{n=1}^\infty$ is generically not a subordinate solution; indeed, subordinacy of $(p_{n-1}(2))_{n=1}^\infty$ is a codimension $1$ condition, as it is typically disrupted by an arbitrary change of a single Jacobi coefficient.

\item If $a_n \le 1$, $b_n \le 0$ for all $n \in \mathbb{N}$, then $p_n(2) \ge n$ for all $n$ (see proof of Prop.~\ref{propeigensolutions2}), so $(p_{n-1}(2))_{n=1}^\infty$ is not subordinate. This includes the case considered in \cite{KreimerLastSimon09} and explains why a consideration of subordinacy is not needed there.

\item The condition that the Dirichlet solution at the critical point shouldn't be the subordinate solution has appeared in related problems; compare, e.g., the work of Simonov~\cite{Simonov12,Simonov16} which considers a class of Wigner--von Neumann type Schr\"odinger operators including $H = -\partial_x^2 + c x^{-\gamma} \sin(2\omega x+\delta)$ for $\gamma>1/2$ and the behavior of the spectral measure around the critical point $\omega^2$ in the interior of the essential spectrum.

\item If the perturbation decays slowly enough, it can be proved from \eqref{decayingsolnasymptotics} that the subordinate solution is $\ell^2$; when that is the case, subordinacy of $(p_{n-1}(2))_{n=1}^\infty$ is equivalent to the presence of a mass point of $\mu$ at $x=2$, i.e., to $\mu(\{2\}) \neq 0$. This will be the situation in the following theorem.
\end{enumerate}
\end{remark}

We present an application of Theorem~\ref{thm2} to sequences of Jacobi coefficients of the form
\begin{equation}\label{anansatz}
a_n = 1 - \sum_{j=1}^J c_j n^{-\tau_j} + o(n^{-1-\tau_1}) \; \forall n\ge N_0\quad\text{or}\quad a_n \equiv 1\; \forall n\ge N_0
\end{equation}
\begin{equation}\label{bnansatz}
b_n = - \sum_{j=1}^K d_j n^{-\sigma_j} + o(n^{-1-\sigma_1})\; \forall n\ge N_0\quad\text{or}\quad b_n \equiv 0\; \forall n\ge N_0
\end{equation}
with $0 < \tau_1 < \dots < \tau_J$, $0 < \sigma_1 < \dots < \sigma_K$, and $c_1, d_1 > 0$.

Pollaczek \cite{Pollaczek1,Pollaczek2,Pollaczek3} considered Jacobi parameters $a_n, b_n$ given by rational functions with numerators and denominators of the same degree (which can be rewritten in the above form, with all exponents negative integers).  Kreimer--Last--Simon \cite{KreimerLastSimon09} considered the case $a_n = 0$, $b_n = - C n^{-\beta}$ and developed some of the techniques we will use. We view \eqref{anansatz}, \eqref{bnansatz} as a natural class of polynomially decaying perturbations; it includes, for instance, linear combinations and products of sequences such as $(n+n_0)^{-\gamma}$, $\gamma >0$. Beyond the intrinsic interest, we will also see that this is crucial for a natural application to orthogonal polynomials on the unit circle. Namely, the more complicated polynomial dependence \eqref{anansatz} arises naturally when considering pure power-law decaying Verblunsky coefficients via the Szeg\H o mapping.

\begin{thm}\label{thm3}
Let $a_n$, $b_n$ be given by \eqref{anansatz}, \eqref{bnansatz}. Denote
\begin{equation}\label{betadefn}
\beta = \min(\sigma_1,\tau_1)
\end{equation}
(with the convention $\tau_1=+\infty$ if $a_n\equiv 1$ for $n\ge N_0$, and $\sigma_1=+\infty$ if $b_n \equiv 0$ for $n\ge N_0$).
\begin{enumerate}[(a)]
\item If $\beta \ge 2$, then $\log f(x) = O(\log(2-x))$, $x \nearrow 2$.
\item If $\beta < 2$ and if the measure $\mu$ does not have a mass point at $x=2$, i.e., $\mu(\{2\})=0$, then $\log f$ has an asymptotic expansion of the form
\begin{equation}\label{asymptotic1000}
\log f(x) = - \sum_{i=1}^{I} Q_i (2-x)^{-\kappa_i} + O(\log(2-x)), \qquad x \nearrow 2,
\end{equation}
with $I \in \mathbb{N}$ and $\kappa_1 > \kappa_2 > \dots > \kappa_I >0$. The leading term is given by
\begin{equation}\label{leadingcoeff}
Q_1 =  C_1^{\frac 1\beta} \frac {\Gamma\left(\frac 1\beta - \frac 12\right) \sqrt \pi } {\Gamma\left(\frac 1\beta  \right)}, \qquad \kappa_1 =  \frac 1\beta - \frac 12,
\end{equation}
where $C_1 > 0$ is such that $C_1 n^{-\beta}$ is the leading term in the expansion of $2 - 2 a_n - b_n$, i.e.
\begin{equation}\label{C1defn}
C_1 = \begin{cases}
2 c_1 & \tau_1 < \sigma_1 \\
2 c_1 + d_1  & \tau_1 = \sigma_1 \\
d_1 & \tau_1 > \sigma_1
\end{cases}
\end{equation}
\end{enumerate}
\end{thm}

\begin{remark}
The proof of existence of the asymptotic expansion \eqref{asymptotic1000} is completely constructive. The constants $Q_i, \kappa_i$ can be computed explicitly in terms of the constants in \eqref{anansatz}, \eqref{bnansatz}, together with combinatorial constants coming from some Taylor expansions, and integrals of the form
\[
\int_0^1 (x^{-\beta} - 1)^{l+1/2} \prod_{j=1}^n \frac{ x^{-\gamma_j} - 1}{ x^{-\beta} - 1} dx
\]
(the leading term has $n=0$, in which case the integral reduces to a Beta function; in general the integrals reduce to a linear combination of Beta functions, possibly with mutually cancelling singularities from the summands). The resulting expressions do not lend themselves to a presentable closed form in the general case, so we don't derive them explicitly; however, they can in principle be recovered in any particular case by following the proof. Moreover, we will see that in some cases of interest (Theorems \ref{thm4} and \ref{thm6} below), it is computationally better not to follow the general method verbatim but to use expansions more tailored to the form of the perturbation.
\end{remark}

Much of the effort in the proof of Theorem~\ref{thm3} goes towards controlling the higher-order terms in \eqref{anansatz}, \eqref{bnansatz}. Lest we neglect the most important special cases, and recalling that \cite{KreimerLastSimon09} described the case $a_n\equiv 1$, $b_n = - C n^{-\beta}$, we consider the case
\begin{equation}\label{pureanansatz}
a_n = 1 - C n^{-\tau}, \qquad b_n = 0
\end{equation}
and compute explicitly the expansion to order $O(\log(2-x))$.

\begin{thm}\label{thm4}
If \eqref{pureanansatz} for some $\tau \in (0,2)$ and $C>0$, then
\begin{equation}\label{asymp1d}
\log f(x) = - \sum_{n=0}^{L-1} T_n (2-x)^{-\frac 1\tau + \frac 12 + n} + O(\log (2-x))
\end{equation}
as $x \nearrow 2$, with $L = \left\lceil \frac 1{\tau} - \frac 12 \right\rceil$ and
\begin{equation}
T_n =  2^{\frac 1\tau  - n} C^{\frac 1\tau} \frac{\Gamma\left(n+\frac 12\right)}{\Gamma\left(\frac 1\tau\right) } \sum_{l=0}^n  \sum_{m=0}^l   \frac{\Gamma\left( m+ \frac 12\right)}{2^{m} m!}  \frac{\Gamma\left( \frac 1\tau - l - \frac 12\right) }{(n-l)! }
\end{equation}
\end{thm}

As the last topic of this paper, we consider measures on the unit circle corresponding to monotone power-law decaying Verblunsky coefficients. Let us recall that, for a nontrivial probability measure $\mu$ on $\partial \mathbb{D}$, applying the Gram--Schmidt process to $1, z, z^2, \dots$ one obtains a sequence of orthonormal polynomials $\varphi_n(z)$, $n=0,1,2,\dots$ such that
\[
\int \overline{\varphi_m(z)} \varphi_n(z) d\mu = \delta_{m,n}.
\]
The polynomials obey the recursion
\begin{equation}\label{Szegorecursion}
\varphi_{n+1}(z) = \frac {z \varphi_n(z) - \bar \alpha_n z^n \overline{\varphi_n(1/\bar z)}}{\sqrt{1-\lvert \alpha_n \rvert^2}}
\end{equation}
for some sequence of Verblunsky coefficients $(\alpha_n)_{n=0}^\infty \in \mathbb{D}^\infty$ \cite{OPUC1}.

Denoting the Lebesgue decomposition of $\mu$ by
\[
d\mu=  w(\theta) \frac{d\theta}{2\pi} + d\mu_\s,
\]
we recall \cite[Section 12.1]{OPUC2} that if $(\alpha_n)_{n=0}^\infty$ have bounded variation and $\alpha_n\to 0$, then $w$ is continuous and strictly positive on $(0,2\pi)$ and $\supp \mu_\s \subset \{1\}$. Our interest will be in the asymptotic behavior of $w(\theta)$ as $\theta \to 0$ when $(\alpha_n)_{n=0}^\infty$ is a suitable polynomially decaying sequence.

The connection with Jacobi parameters is obtained by sieving the Verblunsky coefficients and then applying the Szeg\H o mapping; we will describe the details in Section~\ref{sec4}. In this correspondence, one gets 
\begin{equation}\label{Szegosieving}
a_{n}^2 = (1 - \alpha_{n-2})(1+\alpha_{n-1}), \qquad b_n = 0
\end{equation}
(with the convention $\alpha_{-1}=-1$). Therefore, even starting with purely power-law decaying Verblunsky coefficients
\begin{equation}\label{powerlawverblunsky}
\alpha_n = - \frac{D}{(n+n_0)^\tau}
\end{equation}
one obtains from \eqref{Szegosieving} coefficients $a_n$ of the form \begin{equation}\label{anszegomap}
a_n = 1  - \sum_{k=1}^K \frac{(2k-2)! D^{2k}}{2^{2k-1} k! (k-1)!} n^{-2k\tau}  + \frac{D\tau}2 n^{-1-\tau} + \frac{D^2\tau}2 n^{-1-2\tau} + o(n^{-1-2\tau})
\end{equation}
with $K = \lfloor  \frac 1{2\tau} \rfloor+1$.

We will study a more general class of polynomially decaying Verblunsky coefficients, and then revisit \eqref{powerlawverblunsky} to provide a complete asymptotic expansion for that case.

\begin{thm}\label{thm5}
If $\alpha_n \in \mathbb{R}$ for all $n$ and
\begin{equation}
\alpha_n = - \sum_{i=1}^I D_i n^{-\tau_i} + o(n^{-1-\tau_1}), \qquad n\to\infty,
\end{equation}
for some $0 < \tau_1 < \dots < \tau_I$ and $D_1 > 0$, then
\begin{enumerate}[(a)]
\item If $\tau_1 \ge 1$, then $\log w = O(\log \lvert \theta \rvert)$, $\theta \to 0$.
\item If $\tau_1 < 1$, then $\log w$ has asymptotic behavior of the form
\begin{equation}\label{asymptotic10}
\log w(\theta) = - \sum_{j=1}^J P_j \lvert \theta\rvert^{-\lambda_j} + O(\log \lvert \theta\rvert), \qquad \theta \to 0
\end{equation}
with $J \in \mathbb{N}$ and $\lambda_1 > \lambda_2 > \dots > \lambda_J > 0$. The leading term is given by
\begin{equation}\label{asymp13}
P_1 = 2^{\frac 1{\tau_1} - 1} D_1^{\frac 1{\tau_1}} \frac{\Gamma\left(\frac 1{2\tau_1} - \frac 12\right) \sqrt\pi }{ \Gamma\left(\frac 1{2\tau_1} \right) }, \qquad \lambda_1 = \frac 1{\tau_1} - 1.
\end{equation}
\end{enumerate}
\end{thm}

\begin{thm}\label{thm6}
If $\alpha_n$ are given by \eqref{powerlawverblunsky} for some $\tau \in (0,1)$ and $0 < D < n_0^\tau$, then
\begin{equation}\label{asymp1c}
\log w(\theta) = - \sum_{n=0}^{L-1} P_n \left\lvert \sin \frac \theta 2 \right\rvert^{-\frac 1{\tau} + 1 + 2n}  + O(\log \lvert \theta\rvert)
\end{equation}
with $L = \left\lceil \frac 1{2\tau} - \frac 12 \right\rceil$ and
\begin{equation}
P_n = D^{\frac 1\tau} \frac{\Gamma\left( n+\frac 12\right)}{\Gamma\left(\frac 1{2\tau}\right)} \sum_{l=0}^{n} \frac{\Gamma\left( \frac 1{2\tau} - l - \frac 12 \right)}{(n-l)!}.
\end{equation}
\end{thm}
Of course, by expanding the trigonometric terms, \eqref{asymp1c} can be rewritten into the form \eqref{asymptotic10} if so desired, with $\lambda_j = \frac 1\tau +1 - 2j$. For instance, if
\[
\alpha_n = - \frac {D}{(n+n_0)^{1/2}},
\]
Theorem~\ref{thm6} gives
\[
\log w(\theta) = - \frac{2D^2 \pi}{\lvert \theta\rvert} + O(\log\lvert \theta\rvert), \quad \theta \to 0.
\]

We thank Leonid Golinskii for posing a question answered in this paper and for useful discussions, Barry Simon for showing us the trick of combining sieving with the Szeg\H o mapping used for the OPUC application, Brian Simanek for useful discussions, and an anonymous referee for comments that improved the exposition.

\section{Jacobi matrices with eventually increasing $a_n$ and $b_n$}

In this section, we will prove Theorems~\ref{thm1} and \ref{thm2}. We begin by establishing existence of a subordinate solution at $x=2$ (part (a) of Theorem~\ref{thm1}) and some qualitative properties of the solutions.

\begin{prop}\label{propeigensolutions2}
Assume that $b_n \le 0$ and $a_n \le 1$ for all $n \ge N_0$. Then there exists a solution $v(2) = (v_n(2))_{n=1}^\infty$ of the recurrence relation
\begin{equation}\label{eigen2}
a_{n-1} u_{n-1} + b_n u_n + a_n u_{n+1} = 2 u_n
\end{equation}
such that:
\begin{enumerate}[(a)]
\item $v_n(2) \ge v_{n+1}(2) > 0$ for all $n \ge N_0$; in particular, $v(2) \in \ell^\infty$.

\item For any solution $u=(u_n)_{n=1}^\infty$ of \eqref{eigen2} which is not a multiple of $v$,
\[
\liminf_{n\to\infty} \frac{\lvert u_n\rvert}n > 0.
\]

\item $v(2)$ is subordinate in the sense of Gilbert--Pearson.
\end{enumerate}
\end{prop}

\begin{proof}
If $u_{n} \ge 0$ for some $n>N_0$, then by \eqref{eigen2},
\[
a_{n} (u_{n+1} - u_n) - a_{n-1} (u_n - u_{n-1}) = (2-a_n-a_{n-1} - b_n) u_n \ge 0.
\]
Thus, if $u$ is a solution such that $u_{N+1} > u_N \ge 0$ for some $N \ge N_0$, then it follows by induction that for all $n \ge N_0$,
\[
a_n (u_{n+1} - u_n) \ge a_{n-1} (u_n - u_{n-1}) \text { and } u_n \ge 0.
\]
Therefore,
\[
u_{n+1} - u_n \ge a_n (u_{n+1} - u_n) \ge a_N (u_{N+1} - u_N) > 0\text{ for all }n \ge N_0,
\]
so $\liminf_{n\to \infty} (u_{n+1} - u_n) > 0$ and therefore $\liminf_{n\to \infty} \frac{u_n}n > 0$.

The \eqref{eigen2} has a two-dimensional space of solutions, and a unique solution can be prescribed by setting two consecutive values. For $t\in \mathbb{R}$, denote by $u_n(t)$ that solution of \eqref{eigen2} which obeys $u_{N_0}(t) = 1$ and $u_{N_0+1}(t) = t$. Denote
\[
T = \{ t \in \mathbb{R} \mid  u_n (t) > 0\text{ for all }n > N_0\}
\]
It follows from previous observations that $(1,\infty) \subset T \subset (0, \infty)$, so $s = \inf T$ exists and $s \in [0,1]$.

We prove by contradiction that $u_n(s) \ge u_{n+1}(s)  >0$ for all $n\ge N_0$. Assuming that this was false, let $N$ be the smallest index for which it fails. If $u_{N+1}(s) \le 0 < u_{N}(s)$, then $0 \ge u_{N+1}(s) > u_{N+2}(s)$ so continuity in $t$ would imply that $u_{N+2}(t) < 0$ for $t$ near $s$, contradicting $s = \inf T$. Also, if $u_{N+1}(s) > u_{N}(s) > 0$, then for $t$ near $s$, we would have $u_n(t) > 0$ for $N_0 \le n < N$ and $u_{N+1}(t) > u_{N}(t) > 0$, therefore $u_n(t) > 0$ for all $n \ge N_0$, contradicting $s = \inf T$.

Denoting by $v(2)$ the eigensolution with $v_{N_0}(2)=1$ and $v_{N_0+1}(2)= s$, we have proved that $v_n(2) \ge v_{n+1}(2) > 0$ for all $n\ge N_0$, which proves (a). Representing any other eigensolution as a linear combination of $v(2)$ and the solution with $u_{N_0}=0$, $u_{N_0+1}=1$, we see that any eigensolution which isn't a multiple of $v(2)$ has $\liminf_{n\to\infty} \frac{\lvert u_n\rvert}n > 0$, which is (b). (a) and (b) easily imply (c) by the definition of subordinacy.
\end{proof}

Let us define
\[
x_0 = \max (0, b_{N_0}+2a_{N_0})
\]
and work with $x \in (x_0, 2]$ from now on. Recall the turning point $N(x)$ given by \eqref{Nx}; we allow $x=2$ in which case we write $N(2) = \infty$. Of course, $N(x) \ge N_0$ and $\lim_{x\nearrow 2} N(x) = \infty$. The turning point marks the change from non-oscillatory behavior of eigensolutions to oscillatory behavior. 

Our goal is now to introduce certain uniformly bounded eigensolutions in the non-oscillatory regime for $x<2$.

\begin{prop}\label{propsubordinatex}
There exists a family of eigensolutions $(v_n(x))_{n=1}^\infty$ of \eqref{eigenx} for $x \in (x_0, 2)$ such that
\begin{enumerate}[(a)]
\item $v_{N_0}(x) = 1$ for all $x\in (x_0,2)$;
\item $1 \ge v_n(x) > v_{n+1}(x) \ge 0$ for all $N_0 \le n \le N(x)$;
\item $\lim_{x\nearrow 2} v_n(x) = v_n(2)$ for all $n\in \mathbb{N}$.
\end{enumerate}
\end{prop}

\begin{proof}
For any $x \in (x_0, 2)$, consider the solution $w(x)$ of \eqref{eigenx} given by $w_{N(x)}(x) = 1$, $w_{N(x)+1}(x) = 0$. We claim that for $N_0 \le n \le N(x)$, $w_n > w_{n+1} \ge 0$. This is proved by backwards induction since it holds for $n = N(x)$ and, for any $N_0 < n < N(x)$, if it holds for $n+1$ then
\[
a_{n-1} (w_{n-1} - w_n) - a_n (w_n - w_{n+1}) =  (x - a_{n-1} - b_n - a_n) w_{n} \ge 0
\]
implies that $w_{n-1} - w_{n} > 0$ so it holds for $n$ as well.

Therefore, $w_n$ is a strictly decreasing sequence from $n=N_0$ to $n=N(x)+1$. Taking
\[
v_n(x) = \frac 1{w_{N_0}(x)} w_n(x)
\]
it is immediate that $v(x)$ is a solution of \eqref{eigenx} obeying (a) and (b).

Noting that $v_{N_0+1}(x) \in [0,1]$ for all $x\in (x_0, 2)$, let $x_k \nearrow 2$ be a sequence such that $v_{N_0+1}(x_k)$ converges. Since $v_{N_0}(x)=1$ for all $x$, it follows by induction in $n$ that 
\[
\tilde v_n = \lim_{k\to\infty} v_n(x_k)
\]
converges for all $n$, and the limit is a solution of \eqref{eigenx} with $x=2$ such that $\tilde v_n \in [0,1]$ for all $n\ge N_0$ and $\tilde v_{N_0} = 1$. By Prop.~\ref{propeigensolutions2}(b), boundedness and a normalization condition determine the eigensolution at $x=2$ uniquely, so $\tilde v = v(2)$. Therefore, the limit is independent of subsequence $x_k \nearrow 2$, so by compactness of $[0,1]$, the limit as $x\nearrow 2$ exists.
\end{proof}

We remark that the chosen family of eigensolutions is discontinuous at values of $x$ at which $N(x)$ is discontinuous; however, it is continuous at $x=2$, and this is all that will be needed below.

The next step is a monotonicity statement.

\begin{lemma}\label{lemmaincreasedecrease}
For any $x \in (x_0, 2]$ and $N_0 \le n < N(x)$,
\[
a_n e^{-\gamma_n(x)} \le a_{n+1} e^{-\gamma_{n+1}(x)}, \qquad a_n e^{\gamma_n(x)} \ge a_{n+1} e^{\gamma_{n+1}(x)}.
\]
\end{lemma}

\begin{proof}
Since $a_n, b_n$ are increasing sequences, $\frac{x-b_n}{a_n}$ is a decreasing sequence, so $\gamma_n(x)$ is a decreasing sequence. This implies that $a_n e^{-\gamma_n(x)}$ is an increasing sequence.

Since \eqref{gammanxdef} can be rewritten as
\[
x - b_n = a_n e^{\gamma_n(x)} + a_n e^{-\gamma_n(x)}
\]
and $b_n$ and $a_n e^{-\gamma_n(x)}$ are increasing sequences,
\[
a_n e^{\gamma_n(x)} = x - b_n - a_n e^{-\gamma_n(x)}
\]
is a decreasing sequence.
\end{proof}

At this point let us use $x-b_n = 2 a_n \cosh \gamma_n(x)$ to rewrite \eqref{eigenx} as
\[
u_{n+1} = e^{\gamma_n(x)} u_n + e^{-\gamma_n(x)} u_n - \frac{a_{n-1}}{a_n} u_{n-1}.
\]
Consider the solution of this recurrence with
\begin{equation}\label{initialN0}
u_{N_0-1} = 0, \qquad u_{N_0} = 1
\end{equation}
and introduce for $N_0 \le n \le N(x)+1$
\[
\Phi_n(x) = \exp\left( - \sum_{j=N_0}^{n-1} \gamma_j(x) \right)  u_{n}
\]
with the convention for $\Phi_{N_0-1}$ so that
\begin{equation}\label{initialphi}
\Phi_{N_0-1}(x) = 0, \qquad \Phi_{N_0}(x) = 1.
\end{equation}
Define also for $N_0 \le n \le N(x) + 1$
\[
\phi_n(x) = \Phi_n(x) - \frac{a_{n-1}}{a_{n}} e^{-\gamma_n(x) -\gamma_{n-1}(x)} \Phi_{n-1}(x).
\]
The recursion \eqref{eigenx} can be rewritten for $N_0 \le n \le N(x)$ as
\begin{equation}\label{recursionphi}
\Phi_{n+1}(x) =  \Phi_n (x) + e^{-2\gamma_n(x)} \Phi_n(x) - \frac{a_{n-1}}{a_n} e^{-\gamma_{n-1}(x) - \gamma_n(x)} \Phi_{n-1}(x),
\end{equation}
(note that for $n=N_0$, the undefined quantity $\gamma_{N_0 -1}(x)$ appears in some of the equations, but in a term multiplied by $\Phi_{N_0-1}(x)=0$ so this doesn't present any ambiguity; to avoid undefined quantities let us set $\gamma_{N_0-1}(x)=0$).

In the following proposition we use the argument of \cite{KreimerLastSimon09}, starting from $n=N_0$ instead of $n=1$; the estimates are proved in \cite[Section 3]{KreimerLastSimon09} for the case $b_n \equiv 0$ but, with the additional input of Lemma~\ref{lemmaincreasedecrease}, the arguments work in our level of generality.

\begin{prop}\label{prop24}
Let $(u_n)$ be the solution of \eqref{eigenx}, \eqref{initialN0}. Then
\begin{enumerate}[(a)]
\item For all $N_0 -1 \le n \le N(x)$,  $\Phi_{n+1}(x) \ge \Phi_n(x) \ge 0$.
\item For $N_0 \le n \le N(x)$, $\phi_{n+1}(x) \le \phi_n(x) \le 1$.
\item For all $N_0 \le n \le N(x)$, $\Phi_{n+1}(x) \le \Phi_n (x)+1$.
\item For all $N_0 \le n \le N(x)+1$,
\[
1 \le \Phi_n (x) \le n - N_0 + 1.
\]
\end{enumerate}
\end{prop}

\begin{proof}
For readability, we supress the explicit dependence on $x$ from various functions of $x$ in this proof.

(a) is proved by induction on $n$; the base case $n=N_0-1$ follows from \eqref{initialphi}. For the inductive step, if $\Phi_n \ge \Phi_{n-1} \ge 0$, then rearranging \eqref{recursionphi} and using Lemma~\ref{lemmaincreasedecrease},
\[
\Phi_{n+1} - \Phi_n = \frac{ e^{-\gamma_n}}{a_n} \left( a_n e^{-\gamma_n} \Phi_n - a_{n-1} e^{-\gamma_{n-1}} \Phi_{n-1} \right) \ge 0.
\]

(b) It follows from \eqref{recursionphi} and Lemma~\ref{lemmaincreasedecrease} that
\[
\phi_{n+1} - \phi_n = \frac{e^{-\gamma_{n+1} -2\gamma_{n}} \Phi_n}{a_{n+1}} ( a_{n+1} e^{\gamma_{n+1}} - a_n e^{\gamma_n} ) \le 0.
\]
Therefore, $\phi_n \le \phi_{N_0} = 1$ for all $n\ge N_0$.

(c) follows from (b) as
\[
\Phi_{n+1} = \phi_{n+1} + \frac{a_n}{a_{n+1}} e^{-\gamma_{n+1} - \gamma_n} \Phi_n \le \phi_{n+1} + \Phi_n \le 1 + \Phi_n.
\]

(d) $\Phi_n \ge 1$ follows from (a) and $\Phi_{N_0} = 1$. $\Phi_n \le n- N_0 + 1$ follows by induction from (c).
\end{proof}

\begin{proof}[Proof of Theorem \ref{thm1}]
(a) was proved in Prop.~\ref{propeigensolutions2}, so it remains to prove (b). 
The proof uses estimates on $u_{n+1}$ obtained from Prop.~\ref{prop24},
\[
e^{\sum_{j=N_0}^{n} \gamma_j} \le u_{n+1} \le (n-N_0+2) e^{\sum_{j=N_0}^{n} \gamma_j},
\]
and Wronskian considerations. We recall that the Wronskian of $u, v$,
\[
W = a_n (u_{n+1} v_n - u_n v_{n+1}),
\]
is independent of $n$ by a direct calculation. Taking $n=N_0-1$ we see that $W = a_{N_0-1} v_{N_0-1} > 0$. For any $n$,
\[
a_n u_{n+1} v_n \ge W
\]
so
\[
v_n \ge \frac{W}{a_n} \frac 1{u_{n+1}} \ge \frac W{a_n} \frac 1{n-N_0+2} e^{-\sum_{j=N_0}^n \gamma_j} 
\]
and to have the lower estimate in \eqref{decayingsolnasymptotics} for all $n\ge N_0$ we can take
\[
C_1 = \inf_{n\ge N_0}  \frac W{a_n} \frac n{n-N_0+2}.
\]
By monotonicity of $v$, $u_n v_{n+1} \le u_n v_{n-1}$, so
\[
u_{n+1} v_n - u_n v_{n-1} \le u_{n+1} v_n - u_n v_{n+1} = \frac W{a_n}.
\]
Summing in $n$ from $N_0$ and using $u_{N_0} v_{N_0-1} = W / a_{N_0-1}$ gives
\[
u_{n+1} v_n \le \sum_{j=N_0-1}^{n} \frac W{a_j},
\]
so
\[
v_n \le e^{-\sum_{j=N_0}^n \gamma_j} \sum_{j=N_0-1}^{n} \frac W{a_j}
\]
and to have the upper estimate in \eqref{decayingsolnasymptotics} for all $n\ge N_0$ we can take
\[
C_2 = \sup_{n\ge N_0} \frac 1n \sum_{j=N_0-1}^n \frac W{a_j}.
\]
The constants $C_1, C_2$ are in $(0,\infty)$ because
\[
\lim_{n\to\infty} \frac W{a_n} \frac{n}{n-N_0+2} = W = \lim_{n\to\infty} \frac 1n \sum_{j=N_0-1}^n \frac W{a_j}
\]
(since $a_n \to 1$ and by a Ces\`aro averaging for the second limit).
\end{proof}

The following proposition is the final step in controlling the non-oscillatory regime for Theorem~\ref{thm2}: part (b) is the main part of this proposition, but part (a) is needed to cover some special cases. Define
\begin{equation}\label{gxsum}
g(x) = \sum_{j=N_0}^{N(x)} \gamma_j(x).
\end{equation}

\begin{prop}\label{prop29}
\begin{enumerate}[(a)]
\item There exists $x_1 < 2$ and $C \in (0,\infty)$ such that for all $x \in [x_1,2]$, with $N=N(x)$,
\begin{equation}\label{estimatesq9}
\left( a_N^2 p_{N-1}^2 + p_N^2 \right)^{\pm 1} \le  C N^2 e^{2g}.
\end{equation}
\item If $(p_{n-1}(2))_{n=1}^\infty$ is not a subordinate solution at $x=2$, there exists $x_1 < 2$ and $C \in (0,\infty)$ such that for all $x \in [x_1,2]$, with $N=N(x)$,
\begin{equation}\label{estimatesq1}
\left( \frac{a_N^2 p_{N-1}^2 + p_N^2}{e^{2g}} \right)^{\pm 1} \le  C N^2.
\end{equation}
\end{enumerate}
\end{prop}

\begin{proof}
In this proof we will use a common convention that $C$ stands for different constants in $(0,\infty)$ from one line to the next.

Recall for $x \in (x_0,2]$ the special eigensolutions of \eqref{eigenx}, $u_n(x)$ determined by \eqref{initialN0} and $v_n(x)$ determined by Prop.~\ref{propsubordinatex}. The two solutions are linearly independent, so the eigensolution $(p_{n-1}(x))_{n=1}^\infty$ can be written as their linear combination:
\[
p_{n-1}(x) = \alpha(x) u_n(x) + \beta(x) v_n(x),
\]
for some $\alpha(x), \beta(x) \in \mathbb{R}$. The conditions at two consecutive indices can be written as a $2\times 2$ system for $\alpha, \beta$,
\begin{equation}\label{ml567}
\begin{pmatrix} p_{n} \\ a_n p_{n - 1} \end{pmatrix}
=
T_n \begin{pmatrix} \alpha \\ \beta \end{pmatrix}, \qquad T_n = \begin{pmatrix} u_{n+1} & v_{n+1} \\ a_n u_{n} & a_n v_{n} \end{pmatrix}.
\end{equation}
By Prop.~\ref{propsubordinatex}, $v_n(x)$ is continuous at $x=2$, and the same is obviously true of $p_{n-1}(x)$, $u_n(x)$. By linear independence of $u(x)$ and $v(x)$, the matrix $T_n$ is invertible for all $x$, and since it is continuous at $x=2$, so is its inverse, so $\alpha(x), \beta(x)$ are continuous at $x=2$.

Since the determinant $\det T_n$ is independent of $n$, non-zero and continuous in $x$,
\[
C^{-1} \le \det T_n \le C
\]
in some interval $[x_1,2]$, $x_1 < 2$. Therefore, bounds on the norm of $T_n$ will also yield bounds on the norm of its inverse, since for $2\times 2$ matrices,
\[
\lVert T_n^{-1} \rVert = \frac 1{\lvert \det T_n \rvert} \lVert T_n \rVert.
\]
Since Prop.~\ref{prop24} implies
\begin{equation}\label{ml5}
e^{\sum_{j=N_0}^{n-1} \gamma_j}  \le u_n(x) \le (n - N_0 + 1)e^{\sum_{j=N_0}^{n-1} \gamma_j}
\end{equation}
and $v_n(x)$ are uniformly bounded, we obtain bounds on $\lVert T_n \rVert$ and then $\lVert T_n^{-1}\rVert$,
\[
\lVert T_n^{\pm 1} \rVert \le C n e^{\sum_{j=N_0}^{n} \gamma_j}.
\]
Since $\alpha^2 + \beta^2$ is non-zero and continuous in $x$,
\[
C^{-1} \le \alpha^2 + \beta^2 \le C,
\]
so \eqref{estimatesq9} follows from \eqref{ml567}.

To prove (b), it is necessary to separate cases. Let us first consider the case
\[
\sum_{j=N_0}^\infty \gamma_j(2) < \infty.
\]
In that case, since $\gamma_j(x) < \gamma_j(2)$ for $x \in (x_0,2)$, $g(x)$ is uniformly bounded as $x\nearrow 2$ so \eqref{estimatesq9} implies \eqref{estimatesq1}. Therefore, assume from now on that
\[
\sum_{j=N_0}^\infty \gamma_j(2) = \infty.
\]
Here we use the assumption that $p_{n-1}(2)$ is not a subordinate solution to conclude that $\alpha(2) \neq 0$. Thus, we can pick $N_1 \ge N_0$ such that
\[
e^{\sum_{j=N_0}^{N_1-1} \gamma_j(2)} >  \left \lvert\frac{2\beta(2)}{\alpha(2)} \right\rvert
\]
By continuity, there exists $x_1 \in (x_0,2)$ such that for all $x \in [x_1, 2]$, $N(x) \ge N_1$ and
\[
e^{\sum_{j=N_0}^{N_1-1} \gamma_j(x)} >  \left \lvert \frac{2\beta(x)}{\alpha(x)}  \right\rvert.
\]
Therefore, whenever $x \in [x_1,2]$ and $N_1 \le n \le N(x)+1$, by Prop.~\ref{propsubordinatex},
\begin{equation}\label{betavn}
0 \le \lvert \beta(x) \rvert v_n(x) \le \lvert \beta(x)\rvert \le \frac{\lvert\alpha(x)\rvert}2 e^{\sum_{j=N_0}^{N_1-1} \gamma_j(x)} \le \frac{\lvert\alpha(x)\rvert}2 e^{\sum_{j=N_0}^{n-1} \gamma_j(x)},
\end{equation}
Multiplying \eqref{ml5} by $\alpha(x)$ and combining with \eqref{betavn} by the triangle inequality, we conclude that
\begin{equation}\label{estimatepn}
\frac {\lvert\alpha(x)\rvert}2 e^{\sum_{j=N_0}^{n-1} \gamma_j(x)} \le \lvert p_{n-1}(x) \rvert \le \lvert\alpha(x)\rvert (n - N_0 + \tfrac 32)  e^{\sum_{j=N_0}^{n-1} \gamma_j(x)}.
\end{equation}
Using \eqref{estimatepn} for $n=N(x)$ and $n=N(x)+1$ gives
\begin{equation}\label{estimatesq0}
 C^{-1} e^{2g} \le a_N^2 p_{N-1}^2 + p_N^2 \le  C N^2 e^{2g},
\end{equation}
which implies \eqref{estimatesq1} for $x\in [x_1,2]$.
\end{proof}

Now we need to control the oscillatory part of the solutions. Here we can follow the method of \cite{KreimerLastSimon09}, with modifications to allow both nontrivial sequences of diagonal and off-diagonal Jacobi coefficients; indeed, the proofs of Lemma~\ref{lemma25}, Prop.~\ref{prop27}, and Lemma~\ref{lemma28} are straightforward generalizations of arguments from \cite{KreimerLastSimon09}.

For $n \ge N(x) +1$, define
\[
\kappa_n(x) = \arccos \frac{x-b_n}{2a_n}.
\]
Since $0 < \frac{x - b_n}{2a_n} < 1$, $\kappa_n(x) \in (0,\pi/2)$. We also define
\[
\kappa_\infty(x) = \lim_{n\to\infty} \kappa_n(x) = \arccos \frac x2.
\]

\begin{lemma}\label{lemma25}
As $x \nearrow 2$, with $N = N(x)$,
\[
\kappa_{N+2}(x)^{-2} = O((b_{N+2} - b_{N+1} + a_{N+2} - a_{N+1})^{-1}).
\]
\end{lemma}

\begin{proof}
By definition of $N$,  $\frac{x-b_{N+1}}{a_{N+1}} < 2$, so
\[
2 - 2 \cos \kappa_{N+2} = 2 - \frac{x-b_{N+2}}{a_{N+2}} > \frac{b_{N+2} - b_{N+1} + 2 a_{N+2} - 2 a_{N+1}}{a_{N+2}}.
\]
Since $2 - 2 \cos t \le t^2$ for $t \in (0,\pi/2)$ and $a_{N+2} \ge a_{N_0}$, we conclude that
\[
\kappa_{N+2}(x)^2 \ge \frac 1{a_{N_0}} (b_{N+2} - b_{N+1} + 2 a_{N+2} - 2 a_{N+1}),
\]
which concludes the proof.
\end{proof}

\begin{lemma} For all $x \in (x_0, 2)$ and all $n > N(x)$,
\begin{equation}\label{lemma26a}
\kappa_n \le \kappa_{n+1},
\end{equation}
\begin{equation}\label{lemma26b}
\frac{\cos \kappa_{n+1}}{\cos \kappa_n} \le \frac{a_n}{a_{n+1}} \le 1,
\end{equation}
\begin{equation}\label{lemma26c}
\left\lvert e^{i (\kappa_{n+1} - \kappa_n)} - \frac{a_n}{a_{n+1}} \right\rvert \le  \frac {\kappa_{n+1} - \kappa_n}{\cos \kappa_n}.
\end{equation}
\end{lemma}

\begin{proof}
Since $a_n$, $b_n$ are increasing sequences, $\frac{x-b_n}{2a_n}$ is decreasing, so $\arccos\frac{x-b_n}{2a_n}$ is increasing, which is \eqref{lemma26a}.

\eqref{lemma26b} follows from $a_{n+1} \cos\kappa_{n+1} = \frac{x-b_{n+1}}2 \le \frac{x-b_n}2 = a_n \cos \kappa_n$ and from $a_n \le a_{n+1}$.

By convexity of the absolute value and \eqref{lemma26b},
\[
\left\lvert e^{i (\kappa_{n+1} - \kappa_n)}  - \frac{a_n}{a_{n+1}} \right\rvert \le \max\left(
\left\lvert e^{i (\kappa_{n+1} - \kappa_n)} - \frac{\cos \kappa_{n+1}}{\cos \kappa_n}  \right\rvert,
\left\lvert e^{i (\kappa_{n+1} - \kappa_n)}  - 1 \right\rvert
\right)
\]
The proof of \eqref{lemma26c} is completed by estimating both of the quantities on the right,
\begin{align*}
\left\lvert e^{i (\kappa_{n+1} - \kappa_n)}  - \frac{\cos \kappa_{n+1}}{\cos \kappa_n} \right\rvert
& =  \frac 1{\cos \kappa_n} \left\lvert e^{-i\kappa_n} \cos \kappa_n - \cos \kappa_{n+1} e^{-i\kappa_{n+1}} \right\rvert \\
& =  \frac 1{\cos \kappa_n} \left\lvert \frac{ e^{-2i\kappa_n} - e^{-2i\kappa_{n+1}}} 2\right\rvert \\
& \le \frac{\kappa_{n+1} -\kappa_n}{\cos \kappa_n}
\end{align*}
and
\[
\left\lvert e^{i (\kappa_{n+1} - \kappa_n)}  - 1 \right\rvert = 2 \sin \frac{\kappa_{n+1} - \kappa_n} 2 \le \kappa_{n+1} - \kappa_n \le \frac{\kappa_{n+1} - \kappa_n}{\cos \kappa_n}. \qedhere
\]
\end{proof}

Using $\kappa_n$, \eqref{eigenx} can be rewritten for $n> N(x)$ as
\[
u_{n+1} = 2 u_n \cos \kappa_n  - \frac{a_{n-1}}{a_n} u_{n-1}
\]
Define, for $n > N(x)$,
\[
Y_{n} = u_{n+1} - e^{-i\kappa_{n}} u_{n}.
\]

\begin{prop}\label{prop27}
For any solution $u$ of \eqref{eigenx}, for all $x\in (x_0, 2)$ and all $n > N(x)$,
\begin{equation}\label{prop27a}
\lvert u_n\rvert \le \frac{\lvert Y_{n}\rvert}{\sin \kappa_{n}},
\end{equation}
\begin{equation}\label{prop27b}
\frac 12 \lvert Y_{n}\rvert^2 \le u_n^2 + u_{n+1}^2 \le \frac 2{\sin^2\kappa_n} \lvert Y_{n} \rvert^2,
\end{equation}
\begin{equation}\label{prop27c}
\left\lvert \frac{\lvert Y_{n+1} \rvert}{\lvert Y_{n} \rvert} - 1 \right\rvert \le  \frac{\kappa_{n+1} - \kappa_{n}}{\sin \kappa_{n} \cos \kappa_{n}}.
\end{equation}
\end{prop}

\begin{proof}
\eqref{prop27a} follows from $\Im Y_{n} = u_n \sin \kappa_{n}$ that $\lvert Y_n \rvert \ge \lvert u_n \rvert \sin \kappa_n$.

The first inequality of \eqref{prop27b} follows from $\lvert Y_{n} \rvert \le \lvert u_n \rvert + \lvert u_{n+1} \rvert$ and the arithmetic--quadratic mean. The second inequality follows from
\[
u_n \sin \kappa_n = \Im Y_{n}, \qquad u_{n+1} \sin\kappa_n = \Im (e^{i\kappa_n} Y_{n}).
\]

Since from \eqref{eigenx} one gets
\[
Y_{n+1} - e^{i \kappa_{n+1}} Y_{n} = \left( e^{i(\kappa_{n+1} - \kappa_{n})} - \frac{a_{n}}{a_{n+1}} \right)  u_{n},
\]
using \eqref{lemma26c} gives
\[
\lvert Y_{n+1} - e^{i\kappa_{n+1}} Y_{n} \rvert \le \lvert u_{n} \rvert  \frac{\kappa_{n+1} - \kappa_{n}}{\cos \kappa_{n}} \le \lvert Y_{n} \rvert  \frac{\kappa_{n+1} - \kappa_{n}}{\sin \kappa_{n} \cos \kappa_{n}}
\]
which implies \eqref{prop27c} by dividing by $\lvert Y_n\rvert$ and using $\lvert \lvert z \rvert - 1 \rvert \le \lvert z - e^{i\kappa_{n+1}}\rvert$.
\end{proof}

\begin{lemma} \label{lemma28}
\begin{enumerate}[(a)]
\item The function
\[
q(y) = \sup_{0 < x \le y} \left( \frac 1{\sin x \cos x} - \frac 1x \right)
\]
is finite for $y \in (0,\pi/2)$, and $q(y) = \frac 23 y + O(y^3)$ as $y\downarrow 0$.
\item 
\[
\prod_{n=N+2}^\infty \left( 1 + \frac{ \kappa_{n+1} - \kappa_n }{\sin\kappa_n \cos \kappa_n} \right) \le \frac{\kappa_\infty}{\kappa_{N+2}} \exp( \kappa_\infty q(\kappa_\infty) )
\]
\end{enumerate}
\end{lemma}

\begin{proof}
(a) As $x \to 0$, $\sin x \cos x = x - \frac 23 x^3 + O(x^5)$ so
\[
\frac 1{\sin x \cos x} - \frac 1x = \frac 23 x + O(x^3).
\]
Therefore, this expression extends to a continuous function on any $[0,y]$ with $y < \pi/2$, so $q(y)$ is finite; moreover, $q(y) = \frac 23 y + O(y^3)$ as $y\downarrow 0$.

(b) We use
\begin{align*}
 1 + \frac{ \kappa_{n+1} - \kappa_n}{\sin\kappa_n \cos \kappa_n} & \le 1 + (\kappa_{n+1} - \kappa_n)\left( \frac 1{\kappa_n} + g(\kappa_n) \right) \\
 & \le \frac {\kappa_{n+1}}{\kappa_n} + (\kappa_{n+1} - \kappa_n) g(\kappa_\infty) \\
 & \le \frac{\kappa_{n+1}}{\kappa_n} e^{(\kappa_{n+1} - \kappa_n) g(\kappa_\infty)}
\end{align*}
so taking the product gives
\[
\prod_{n=N+2}^\infty  \left( 1 + \frac{ \kappa_{n+1} - \kappa_n }{\sin\kappa_n \cos \kappa_n} \right) \le \frac{\kappa_\infty}{\kappa_{N+2}} \exp( (\kappa_\infty - \kappa_{N+2}) q(\kappa_\infty) ).
\]
Since $\kappa_{N+2} > 0$, this concludes the proof.
\end{proof}

\begin{prop}\label{prop212}
With $h=h(x)$ defined in \eqref{eq19} and writing $N=N(x)$,
\begin{equation}\label{neweq}
\sup_{n \ge N+2} \left(\frac{p_{n-1}^2 + p_n^2}{p_{N+1}^2 + p_{N+2}^2}\right)^{\pm 1} = O( \kappa_{N+2}^{-4} e^{2(2-x)}), \qquad x \nearrow 2.
\end{equation}
\end{prop}

\begin{proof}
Applying \eqref{prop27c} to the solution $(p_{n-1}(x))$ together with Lemma~\ref{lemma28}, we conclude that for all $n \ge N + 2$,
\[
\left \lvert \frac{  Y_n}{ Y_{N+2}} \right\rvert^{\pm 1} \le \frac{\kappa_\infty}{\kappa_{N+2}} \exp(\kappa_\infty q(\kappa_\infty)).
\]
Prop.~\ref{prop27}(b) gives
\[
 \frac {\sin^2 \kappa_{N+2}}4 \left \lvert \frac{  Y_n}{ Y_{N+2}} \right\rvert^{2} \le 
 \frac{p_{n-1}^2 + p_n^2}{p_{N+1}^2 + p_{N+2}^2}  \le \frac 4{\sin^2 \kappa_n} \left \lvert \frac{  Y_n}{ Y_{N+2}} \right\rvert^{2}
\]
and, since $\sin \kappa_{N+2} \le \sin\kappa_n$, we can combine the previous inequalities as
\begin{align*}
\left( \frac{p_{n-1}^2 + p_n^2}{p_{N+1}^2 + p_{N+2}^2} \right)^{\pm 1}
& \le \frac 4{\sin^2 \kappa_{N+2}} \left \lvert \frac{  Y_n}{ Y_{N+2}} \right\rvert^{\pm 2} \\
& \le  \frac 4{\sin^2 \kappa_{N+2}} \frac{\kappa_\infty^2}{\kappa_{N+2}^2} \exp(2\kappa_\infty q(\kappa_\infty))
\end{align*}
%Using Lemma~\ref{lemma25} and \eqref{estimatesq}, this becomes
%\[
%\left(\frac{p_{n-1}^2 + p_n^2}{e^{2g}}\right)^{\pm 1} \le C N^2 \frac{4 \kappa_\infty^2}{\sin^4 \kappa_{N+2}} \exp(2 \kappa_\infty g(\kappa_\infty))
%\]
Since $\kappa_\infty = \arccos \frac x2 \to 0$ as $x \nearrow 2$, by Lemma~\ref{lemma28}(a),
\[
\kappa_\infty q(\kappa_\infty) = \frac 23 \kappa_\infty^2 + O(\kappa_\infty^4) = \frac 23 (2-x) + O((2-x)^2), \qquad x\nearrow 2.
\]
Finally using $\sin \kappa_{N+2} \ge \frac 2\pi \kappa_{N+2}$ completes the proof.
%and, $\kappa_\infty g(\kappa_\infty) \le \kappa_\infty^2$ for $x$  sufficiently close to $2$.
\end{proof}

Before we present the proof of Theorem~\ref{thm2}, we recall that the recurrence relation \eqref{eigenx} can equivalently be written in matrix form as
\[
\begin{pmatrix}
u_{n+1} \\
a_{n} u_n
\end{pmatrix}
= A_n
\begin{pmatrix}
u_{n} \\
a_{n-1} u_{n-1}
\end{pmatrix},
\qquad 
A_n = \begin{pmatrix}
\frac{x-b_n}{a_n} & - \frac 1{a_n} \\
a_n & 0
\end{pmatrix}
\]
A standard observation about these transfer matrices is that $\det A_n = 1$ so $\lVert A_n^{-1} \rVert = \lVert A_n \rVert$. In particular, this norm is uniformly bounded for $x\in [-2,2]$ and $n\ge N_0$ by boundedness of $b_n$ and by $a_n  \in [a_{N_0}, 1]$.

\begin{proof}[Proof of Theorem \ref{thm2}]
We continue to write $N=N(x)$ for readability. Since
\[
\begin{pmatrix}
p_{N+2} \\
a_{N+2} p_{N+1}
\end{pmatrix}
= A_{N+2} A_{N+1}
\begin{pmatrix}
p_{N} \\
a_{N} p_{N-1}
\end{pmatrix},
\]
uniform boundedness of norms of $A_{n}$ implies that
\[
\left( \frac{a_{N+2}^2 p_{N+1}^2 + p_{N+2}^2}{a_{N}^2 p_{N-1}^2 + p_{N}^2} \right)^{\pm 1} \le C
\]
 and therefore, by Prop.~\ref{prop29}(b), also
\begin{equation}\label{estimatesq}
\left( \frac{p_{N+1}^2 + p_{N+2}^2}{e^{2g}} \right)^{\pm 1} \le C N^2.
\end{equation}
Combining this with Prop.~\ref{prop212} and Lemma~\ref{lemma25} gives
\[
\sup_{n\ge N+2} \left(\frac{p_{n-1}^2 + p_n^2}{e^{2g}}\right)^{\pm 1} = O( e^{2 h}), \qquad x \nearrow 2,
\]
with $h(x)$ defined in \eqref{eq19}. By \cite[Corollary 1.3]{KreimerLastSimon09} (as a corollary of Carmona's formula, see also \cite{Carmona83,LastSimon99,KrutikovRemling01,Simon07}), this implies that the density $f(x)$ of the spectral measure obeys
\[
\lvert \log f(x) + 2 g(x) \rvert \le 2 h(x) + O(1), \qquad x \nearrow 2,
\]
which concludes the proof.
\end{proof}

\section{Polynomially decaying Jacobi perturbations}\label{sec3}

In this section we consider sequences $a_n$, $b_n$ given by \eqref{anansatz}, \eqref{bnansatz} and prove Theorem~\ref{thm3}.  We will use notation $f \sim g$ to denote that asymptotically, $C^{-1} g \le f \le C g$ for some $C \in (0,\infty)$.

\begin{lemma} If sequences $a_n$, $b_n$ obey \eqref{anansatz}, \eqref{bnansatz}, and $\beta < +\infty$, then as $x \nearrow 2$, writing $N = N(x)$,
\begin{equation}\label{Nxasymp}
N \sim (2-x)^{-1/\beta},
\end{equation}
\begin{equation}\label{badiff}
b_{N+2}-b_{N+1}+ a_{N+2} - a_{N+1} \sim  (2-x)^{1 + 1/\beta},
\end{equation}
\begin{equation}\label{hxasymp}
e^{h(x)} \sim (2-x)^{-1/2 - 2/\beta}.
\end{equation}
\end{lemma}

\begin{proof}
If $a_n$ is not eventually constant, then $1 - a_n \sim n^{-\tau_1}$, and if $b_n$ is not eventually constant, then $-b_n \sim n^{-\sigma_1}$. Thus, $2 - 2a_n - b_n \sim n^{-\beta}$, so \eqref{Nxasymp} follows.

Similarly, $a_{N+2} - a_{N+1} \sim N^{-1-\tau_1}$ and $b_{N+2}-b_{N+1} \sim N^{-1-\sigma_1}$ imply
\[
b_{N+2}-b_{N+1}+ a_{N+2} - a_{N+1} \sim N^{-1-\beta}
\]
and combining with \eqref{Nxasymp} gives \eqref{badiff}. Combining \eqref{Nxasymp} and \eqref{badiff} into \eqref{eq19} gives \eqref{hxasymp}.
\end{proof}

With this, we can already get the case $\beta \ge 2$ out of the way.

\begin{proof}[Proof of Theorem~\ref{thm3}(a)]
We begin by briefly addressing the case $\beta = +\infty$ (when both sequences $a_n, b_n$ are eventually constant): that case is easy but doesn't fit in the general arguments that follow. Let us note a uniform bound
\[
(p_{N_0-1}^2 + p_{N_0}^2)^{\pm 1} \le C
\]
which holds for $x\in [x_1,2]$ by continuity. Multiplying this by \eqref{neweq} from Prop.~\ref{prop212} and using $\kappa_{N+2}(x)  = \arccos \frac x2$
we obtain 
\[
\sup_{n\ge N_0} (p_{n-1}^2 + p_n^2)^{\pm 1} = O((2-x)^2 e^{2(2-x)})
\]
and therefore $f(x) = O(\log(2-x))$.

From now on we assume $2\le \beta < +\infty$. By Prop.~\ref{prop29}(a), with $g$ given by \eqref{gxsum},
\[
( p_{N+1}^2 + p_{N+2}^2 )^{\pm 1} \le C N^2 e^{2g(x)}
\]
Multiplying this with \eqref{neweq} from Prop.~\ref{prop212} and using Lemma~\ref{lemma25}, we see that 
\[
\sup_{n\ge N+2} (p_{n-1}^2 + p_n^2)^{\pm 1} = O(e^{2 h(x) + 2 g(x)})
\]
and therefore $\log f(x) = O(h(x) + g(x))$.

In the case $\beta \ge 2$, we have $a_n, b_n = O(n^{-2})$ so
\[
\gamma_n(x) = \arccosh \frac{x-b_n}{2a_n} \le \arccosh\frac{2-b_n}{2a_n} = O(n^{-1})
\]
and therefore
\[
g(x) = \sum_{n=N_0}^{N(x)} \gamma_n(x) = O( \log N(x)) = O(\log (2-x)).
\]
Since \eqref{hxasymp} implies $h(x) = O(\log (2-x))$, the proof is completed.
\end{proof}

We now turn our attention to the more interesting case $\beta < 2$. The function $h(x)$ is still described by \eqref{hxasymp} so the remainder of our effort is directed at the asymptotic behavior of the function $g(x)$ given by \eqref{gxsum}.

Denote
\[
\delta = 2-x
\]
and define
\[
\tilde B_n = \frac{2 - 2a_n - b_n}{a_n}, \qquad \tilde A_n = a_n^{-1} - 1,
\]
in order to write
\[
\frac{x-b_n}{2a_n} = 1 +  \frac{2-2a_n - b_n}{2a_n} - \frac{\delta}{2a_n} = 1 + \frac 12 ( \tilde B_n -  \delta \tilde A_n - \delta).
\]
Denoting also
\[
F(z) = \begin{cases} \arccosh \left(1 + \frac z2\right)  & z > 0 \\
0 & z \le 0 
\end{cases}
\]
allows us to write the function $g(x)$ as
\[
g(x) = \sum_{n=N_0}^\infty F(\tilde B_n -  \delta \tilde A_n - \delta)
\]
By algebraic manipulations, $\tilde B_n$ and $\tilde A_n$ can be written in the form
\[
\tilde B_n = B_n + o(n^{-1-\beta}), \qquad \tilde A_n = A_n + o(n^{-1-\beta})
\]
where
\begin{align}
B_n & = C_1 n^{-\beta} + \sum_{l=2}^L C_l n^{-\beta_l} \label{Bnrep} \\
A_n & = D_1 n^{-\beta} + \sum_{l=2}^L D_l n^{-\beta_l} \label{Anrep}
\end{align}
with $0 < \beta < \beta_2 < \dots < \beta_L$ and $C_1>0$, $D_1\ge 0$.

The following lemma will allow us to ignore the $o(n^{-1-\beta})$ remainder terms.

\begin{lemma}\label{lemma32b}
\begin{equation}\label{gxsum2}
g(x) = \sum_{n=N_0}^\infty F(B_n -  \delta A_n - \delta) + O(1), \qquad x\nearrow 2.
\end{equation}
\end{lemma}

\begin{proof}
Plugging in $n\pm 1$ for $n$ and expanding gives
\begin{align*}
B_{n\pm 1} & = B_n \mp  C \beta  n^{-1-\beta} + o(n^{-1-\beta}) \\
A_{n \pm 1} & = A_n \mp  D \beta n^{-1-\beta} + o(n^{-1-\beta})
\end{align*}
so
\[
B_{n\pm 1} - \delta A_{n\pm 1} = B_n - \delta A_n \mp (C - \delta D) n^{-1-\beta} + o(n^{-1-\beta}).
\]
Thus, there exists $N_2$ such that for all $n \ge N_2$ and all $\delta \in (0,\frac{D}{2C})$,
\[
B_{n+1} -  \delta A_{n+1}  \le \tilde B_n -  \delta \tilde A_n  \le B_{n-1} -  \delta A_{n-1} .
\]
Since the function $F$ is monotone increasing on $\mathbb{R}$, this implies
\[
\sum_{n=N_2+1}^\infty F(B_n -  \delta A_n - \delta) \le g(x) \le \sum_{n=N_2-1}^\infty F(B_n -  \delta A_n - \delta)
\]
Since individual terms of the sum are uniformly bounded, this concludes the proof.
\end{proof}

To proceed further, we redefine $N(x)$ as
\[
N(x) = \max \{ n : B_n - \delta A_n - \delta \ge 0 \}
\]
so that the sum in \eqref{gxsum2} terminates at $N(x)$, and we use the series expansion
\[
F(z) = \sum_{l=0}^{l_0-1} c_l z^{l+1/2} + O(z^{l_0 + 1/2}), \qquad z \downarrow 0,
\]
as observed in \cite{KreimerLastSimon09}; in particular, $c_0=1$. Taking $l_0 = \lfloor \frac 1{\beta} - \frac 12 \rfloor+1$, we have
\[
(B_n - \delta A_n - \delta)^{l_0 + 1/2} = O( n^{-\beta(l_0 + 1/2)}), \qquad n \to \infty
\]
and $\beta (l_0 + 1/2) > 1$ so
\begin{align*}
g(x) & = \sum_{n=N_0}^{N(x)} F(B_n - \delta A_n - \delta) + O(1) \\
& =  \sum_{n=N_0}^{N(x)} \sum_{l=0}^{l_0-1} c_l (B_n - \delta A_n - \delta )^{l+1/2}  + O(1), \qquad x\nearrow 2.
\end{align*}
Therefore, $g(x)$ can be written with error $O(1)$ as a linear combination of sums of fractional powers of $B_n - \delta A_n - \delta$,
\[
g(x) = \sum_{l=0}^{l_0-1} c_l S_l(x) + O(1), \quad x \nearrow 2,
\]
where
\[
S_l (x) = \sum_{n=N_0}^{N(x)} (B_n - \delta A_n - \delta)^{l+\frac 12}.
\]
If $l_0 -1 = 1/\beta - 1/2$, then
\[
S_{l_0}(x) \le  \sum_{n=N_0}^{N(x)} O( n^{-1}) = O(\log N(x)) = O(\log(2-x)), \quad x \nearrow 2.
\]
It remains to describe the asymptotics of $S_l(x)$ for $l < 1/\beta - 1/2$.

We note that $S_l(x)$ is an increasing function of $x$ and consider that function on the subsequence $(x_N)_{N=N_0}^\infty$ determined so that $\delta_N = 2 - x_N$ obeys
\[
B_N - \delta_N A_N - \delta_N = 0.
\]
In this definition we think of $N$ as the independent variable and $\delta_N$ as a sequence in $N$; we will continue this point of view for a while, as we work with objects that only depend on $x$ through $N=N(x)$. We now construct an asymptotic expansion of $N$ in terms of $\delta_N$.

\begin{lemma}\label{lemmaNdeltaN}
For any exponent $s>0$, $N^s$ has an expansion in terms of $\delta_N$ of the form
\begin{equation}\label{NdeltaNrep0}
N^s =  \sum_{i=1}^I \tilde C_i \delta_{N}^{-\xi_i} + O(1), \qquad N \to \infty,
\end{equation}
with $s/\beta = \xi_1 > \xi_2 > \dots > \xi_I >0$, $\tilde C_1 = C_1^{s/\beta}$, and some $\tilde C_2, \dots, \tilde C_I \in \mathbb{R}$.
\end{lemma}

\begin{proof}
Let $\epsilon = \min (\beta_2 - \beta, \beta)$.  We will prove by reverse induction that for every nonnegative integer $m$, there is an asymptotic expansion of the form
\begin{equation}\label{NdeltaNrep}
N^s =  \sum_{i=1}^I \tilde C_i \delta_N^{-\xi_i} + O(N^{m\epsilon}), \qquad N \to \infty.
\end{equation}
For $m \ge s/\epsilon$ this is trivial, which provides the basis of induction. Let us assume there is such an expansion for some $m$. Let us first observe that, by $\delta_N = \frac{B_N}{A_N+1}$, $\delta_N$ can be written as a finite linear combination of negative powers of $N$ to any order $O(N^c)$ and
\[
\delta_N = \frac{B_N}{A_N + 1} = C_1 N^{-\beta} ( 1 + O(N^{-\epsilon}))
\]
and therefore for any $t$,
\begin{equation}\label{eqrev43}
\delta_N^{-t/\beta} = C_1^{-t/\beta} N^{t} (1 + O(N^{-\epsilon})).
\end{equation}
Expanding each $\delta_N^{-\xi_i}$ occuring in \eqref{NdeltaNrep} as a linear combination of powers of $N$ to order $O(N^{(m-1)\epsilon})$, the sum from \eqref{NdeltaNrep} can be expanded as
\[
\sum_{i=1}^I \tilde C_i \delta_N^{-\xi_i} = \sum_{j=0}^J c_j N^{t_j} + O(N^{(m-1)\epsilon}).
\]
Comparing this equality with the inductive hypothesis \eqref{NdeltaNrep}, we see that one of the terms matches $N^s$, let's say $c_0 N^{t_0} = N^s$, and that all other terms have to be $O(N^{m\epsilon})$, which implies $t_j \le m\epsilon$ for all $j \ge 1$. Then using \eqref{eqrev43} with $t=t_j$,
\begin{align*}
\sum_{i=1}^I \tilde C_i \delta_N^{-\xi_i}  & = N^s + \sum_{j=1}^J c_j \left(C_1^{t/\beta} \delta_N^{-t_j/\beta} + O(N^{t_j-\epsilon}) \right) + O(N^{(m-1)\epsilon}) \\
& = N^s + \sum_{j=1}^J c_j C_1^{t/\beta} \delta_N^{-t_j/\beta} + O(N^{(m-1)\epsilon})
\end{align*}
Moving all powers of $\delta_N$ to one side gives a representation of $N^s$ of the form \eqref{NdeltaNrep} to order $O(N^{(m-1)\epsilon})$, completing the inductive step.
\end{proof}

For the above choice of $x_N = 2 - \delta_N$, we can rewrite the sum $S_l$ as
\[
S_l(x_N) = \sum_{n=N_0}^N [(B_n - B_N) - \delta_N (A_n - A_N)]^{l+1/2}.
\]
The following lemma will allow us to identify the leading term in the summand.

\begin{lemma}\label{lemma32}
For any $\epsilon > 0$ and $0 < \beta < \sigma$, there exists $N_1>0$ such that for all $N, n$ with $N_1 \le  n < N$,
\begin{equation}\label{inequality1}
n^{-\sigma} - N^{-\sigma} < \epsilon ( n^{-\beta} - N^{-\beta} ).
\end{equation}
\end{lemma}

\begin{proof}
We begin by noting that the function
\[
h(x) = \epsilon x^{-\beta} - x^{-\sigma}
\]
has
\[
h'(x) = \sigma x^{-\sigma - 1} - \epsilon \beta x^{-\beta -1} = ( \sigma - \epsilon x^{\sigma - \beta} ) x^{-\sigma -1}
\]
so, if we pick $N_1$ by the condition $\sigma = \epsilon N_1^{\sigma - \beta}$, then $h'(x) < 0$ for $x > N_1$. Therefore, for any $N, n$ with $N_1 \le n < N$, $h(n) > h(N)$, which can be rearranged into the form \eqref{inequality1}.
\end{proof}

\begin{cor}
For any $\epsilon > 0$, there exists $N_1$ such that for $N \ge N_1$, for any $n$ with $N_1 \le n \le N-1$,
\[
\left\lvert [(B_n - B_N) - \delta_N(A_n - A_N)] - C_1 (n^{-\beta} - N^{-\beta}) \right\rvert < \epsilon (n^{-\beta} - N^{-\beta}).
\]
where $C_1$ is the constant from \eqref{Bnrep}.
\end{cor}

\begin{proof}
Lemma~\ref{lemma32} applies to all the terms from \eqref{Bnrep} and \eqref{Anrep}, except for the $D_1 n^{-\beta}$ term from \eqref{Anrep} if $D_1 \neq 0$. However, in the quantity  $(B_n - B_N) - \delta(A_n - A_N)$, that term is multiplied by $\delta$, which we can make arbitrarily small by making $N$ sufficiently large. This completes the proof.
\end{proof}

This will enable us to apply the power series expansion
\begin{equation}\label{powerexp}
(z_0 + z)^{l+1/2} = z_0^{l+1/2} \sum_{k=0}^\infty \binom{l+1/2}{k} \left( \frac{z}{z_0} \right)^k
\end{equation}
with $z_0 = C_1(n^{-\beta} - N^{-\beta})$ and $z =(B_n - B_N) - \delta_N(A_n - A_N)$. To know which terms to keep individually and which to collect into a remainder term, we need the following lemmas.

\begin{lemma}\label{lemma34}
If $0 < \beta < \gamma$ and $N>0$, then
\begin{equation}\label{xfunclemma}
t\mapsto \frac{t^{-\gamma} - N^{-\gamma}}{t^{-\beta} - N^{-\beta}}
\end{equation}
is a strictly decreasing function of $t \in (0,N)$.
\end{lemma}

\begin{proof}
By substituting $y=t^{-\beta}$, $a=N^{-\beta}$, $c=\gamma/\beta$, the statement assumes the equivalent form that if $c>1$, the function
\begin{equation}\label{yfunclemma}
y \mapsto \frac{y^c - a^c}{y-a}
\end{equation}
is strictly increasing on $(a,\infty)$. The derivative of this function is
\[
\frac{(c-1)y^c - c y^{c-1} a + a^c}{(y-a)^2}
\]
which is strictly positive on $(a,\infty)$ by the weighted arithmetic--geometric mean inequality:
\[
y^{c-1} a  < \frac{c-1}c y^c + \frac 1c a^c.
\]
Thus, \eqref{yfunclemma} is strictly increasing on $(a,\infty)$, and \eqref{xfunclemma} strictly decreasing on $(0,N)$.
\end{proof}

\begin{lemma}\label{lemma36}
Let $\gamma_1, \dots, \gamma_J > \beta$ and define
\[
\lambda =  \left(l + \frac 12\right) \beta+ \sum_{j=1}^J (\gamma_j - \beta).
\]
Then the integral
\begin{equation}\label{intlim}
I = \int_0^1 (x^{-\beta} - 1)^{l+1/2} \prod_{j=1}^J \frac{ x^{-\gamma_j} - 1}{ x^{-\beta} - 1} dx
\end{equation}
is finite if and only if $\lambda < 1$. Moreover, as $N \to \infty$,
\begin{equation}\label{sumintlim2}
\sum_{n=N_1}^{N-1} (n^{-\beta} - N^{-\beta})^{l+1/2} \prod_{j=1}^J \frac{n^{-\gamma_j} - N^{-\gamma_j }}{n^{-\beta} - N^{-\beta}} = \begin{cases} 
N^{1-\lambda} I + O(1) & \lambda < 1 \\
O(\ln N) & \lambda = 1  \\
O(1) & \lambda > 1
\end{cases}.
\end{equation}
\end{lemma}

\begin{proof}
The integrand is a continuous function on $(0,1)$ which converges to $0$ as $x\to 1$ and behaves asymptotically as $x^{-\lambda}$ as $x \to 0$; therefore, the integral is finite if and only if $\lambda <1$.

Uniform boundedness of the summand and its monotonicity (by Lemma~\ref{lemma34}) imply that the sum is a good approximation of the integral,
\begin{align*}
& \sum_{n=N_1}^{N-1} (n^{-\beta} - N^{-\beta})^{l+1/2} \prod_{j=1}^J \frac{n^{-\gamma_j} - N^{-\gamma_j }}{n^{-\beta} - N^{-\beta}}   \\
& = \int_{N_1}^{N} (t^{-\beta} - N^{-\beta})^{l+1/2} \prod_{j=1}^J \frac{t^{-\gamma_j} - N^{-\gamma_j }}{t^{-\beta} - N^{-\beta}} dt + O(1), \qquad N\to \infty \\
& = N^{1-\lambda} \int_{N_1/N}^{1} (x^{-\beta} - 1)^{l+1/2} \prod_{j=1}^J \frac{x^{-\gamma_j} - 1}{x^{-\beta} - 1} dx + O(1), \qquad N\to \infty
\end{align*}
where in the last line we used the substitution $x=t/N$.

If $\lambda < 1$, subtracting $N^{1-\lambda} I$ gives
\begin{align*}
N^{1-\lambda} \int_0^{N_1/N} (x^{-\beta} - 1)^{l+1/2} \prod_{j=1}^J \frac{x^{-\gamma_j} - 1}{x^{-\beta} - 1} dx 
& \sim N^{1-\lambda} \int_0^{N_1/N} x^{-\lambda} dx \\
& \sim N^{1-\lambda} \left(\frac{N_1}N \right)^{1-\lambda} \\
& \sim 1
\end{align*}
which proves the first case in \eqref{sumintlim2}. If $\lambda \ge 1$, the integral diverges as the lower limit goes to $0$, so the asymptotics of the integral is determined by the asymptotics of the integrand at the singular point,
\[
\int_{N_1/N}^{1} (x^{-\beta} - 1)^{l+1/2} \prod_{j=1}^J \frac{x^{-\gamma_j} - 1}{x^{-\beta} - 1} dx  \sim \int_{N_1/N}^1 x^{-\lambda} dx
\]
which implies
\[
\int_{N_1/N}^{1} (x^{-\beta} - 1)^{l+1/2} \prod_{j=1}^J \frac{x^{-\gamma_j} - 1}{x^{-\beta} - 1} dx  \sim 
\begin{cases}
\left( \frac{N_1}N \right)^{1-\lambda} & \lambda > 1 \\
- \ln \frac{N_1}N & \lambda = 1
\end{cases}
\]
which implies \eqref{sumintlim2}.
\end{proof}

\begin{proof}[Proof of Theorem~\ref{thm3}(b)]
We begin by verifying the non-subordinacy condition. By Theorem~\ref{thm1}, the subordinate solution $(v_n)_{n=1}^\infty$ has asymptotic behavior
\[
v_n  = - \sum_{n=N_0}^N \arccosh \frac{2-b_n}{2a_n} + O( \log N)
\]
and since $\frac{2-b_n}{2a_n} - 1 \sim  n^{-\beta}$ and $0 < \beta < 2$, we have $\Gamma_n \sim n^{-\beta/2}$ so
\[
\log v_n \sim - N^{1-\frac \beta 2}
\]
Therefore, $v_n$ decays superpolynomially, so $(v_n)_{n=1}^\infty \in \ell^2$. Thus, $\mu(\{2\}) = 0$ implies that the the $\ell^2$ solution is not the Dirichlet solution, so the Dirichlet solution is not subordinate. The conditions of Theorem~\ref{thm2} are therefore satisfied.

As described above, the function $g(x)$ can be expressed with error $O(\log \delta)$ as a linear combination of power sums $S_l(x)$, and we will begin by considering those sums at points $x=x_N$.

Using \eqref{Bnrep}, \eqref{Anrep}, and the power series expansion \eqref{powerexp} with $z_0 = n^{-\beta} - N^{-\beta}$ and $z =(B_n - B_N) - \delta_N(A_n - A_N)$, there will be finitely many terms corresponding to $\lambda < 1$, since each additional factor contributes at least $\beta_2 - \beta$ towards the exponent $\lambda$. Therefore, applying Lemma~\ref{lemma36} to each of the terms, we get a finite sum plus a logarithmic error term,
\[
S_l(x_N) = \sum_{i=1}^{i_l} Q_{i,l}  I_{i,l} \delta_N^{m_{i,l}} N^{1-\lambda_{i,l}} + O(\log N)
\]
where $0 < \lambda_{1,l} < \dots < \lambda_{i_l,l} < 1$, $Q_{i,l}$ is determined by the coefficients in \eqref{Bnrep}, \eqref{Anrep}, \eqref{powerexp}, $m_{i,l}$ are nonnegative integers, and $I_{i,l}$ are integrals of the form \eqref{intlim}. Using the asymptotic expansion of $N^{1-\lambda_{i,l}}$ \eqref{NdeltaNrep0}, one obtains an asymptotic expansion
\[
S_l(x_N) = \sum_{i=1}^{\tilde i_l} \tilde Q_{i,l} \delta_N^{-\kappa_i}  + O(\log N)
\]
where $\kappa_i < 1/ \beta$ because the asymptotic expansion \eqref{NdeltaNrep0} of $N^{1-\lambda_{i,l}}$ starts at order $\delta_N^{(1-\lambda_{i,l})/\beta}$.

Summing in $l$ from $0$ to $l_0-1$ we obtain
\begin{equation}\label{asymp300}
g(x_N) = \sum_{i=1}^{\tilde I} Q_{i}  \delta_N^{-\kappa_i}  + O(\log N)
\end{equation}
with $1/\beta > \kappa_1 > \dots > \kappa_{\tilde I} > 0$. 

Since
\[
\delta_{N+1}^{-1/\beta} - \delta_N^{-1/\beta} = C^{-1/\beta}  + O(n^{-\epsilon}) = O(1),
\]
the same is true for $\delta_{N+1}^{-t} - \delta_N^{-t}$ for any $0 < t < 1/\beta$. Therefore, $g(x_N)$, $g(x_{N+1})$ obey the same asymptotics to order $O(\log (2-x))$, and since $g(x_N) \le g(x) \le g(x_{N+1})$ if $N=N(x)$, \eqref{asymp300} implies that $g(x)$ also obeys the same asymptotics with $\delta_N$ replaced by $\delta$, which proves \eqref{asymptotic1000}.

The leading term in the asymptotic expansion is obtained from the sum \eqref{sumintlim2} with $l=0$, $J=0$, $m=0$, and
\[
Q_{1} I_{1} N^{1-\lambda_{1}} = C_1^{1/2} N^{1-\beta/2} \int_0^1 (x^{-\beta} - 1)^{1/2} dx.
\]
This integral has appeared in \cite{KreimerLastSimon09}, where the integral was reduced to a Beta function by the substitution $x^\beta = u$, so that
\[
\int_0^1 (x^{-\beta} - 1)^{1/2} dx = \frac 1\beta \int_0^1 u^{\frac 1\beta - \frac 32} (1-u)^{\frac 12} du = \frac {\Gamma\left(\frac 1\beta - \frac 12\right) \Gamma\left(\frac 32\right)} {\beta \Gamma\left(\frac 1\beta + 1\right)} = \frac {\Gamma\left(\frac 1\beta - \frac 12\right) \sqrt\pi} {2 \Gamma\left(\frac 1\beta \right)}
\]
The leading term in the asymptotic expansion \eqref{NdeltaNrep0} of $N^{1-\beta/2}$ is $C_1^{\frac 1\beta - \frac 12} \delta_N^{\frac 12 - \frac 1\beta}$ so the leading term in the asymptotic expansion of $g(x)$ is
\[
C_1^{\frac 12} C_1^{\frac 1\beta - \frac 12} \delta^{\frac 12 - \frac 1\beta}  \frac {\Gamma\left(\frac 1\beta - \frac 12\right) \sqrt\pi} {2 \Gamma\left(\frac 1\beta \right)} 
 = C_1^{\frac 1\beta}  \frac {\Gamma\left(\frac 1\beta - \frac 12\right) \sqrt\pi} {2 \Gamma\left(\frac 1\beta \right)} \delta^{\frac 12 - \frac 1\beta}
\]
which concludes the proof.
\end{proof}

\begin{proof}[Proof of Theorem~\ref{thm4}]
We write
\[
a_n = 1 - \epsilon_n, \qquad \epsilon_n = C n^{-\tau},
\]
and we introduce the function
\[
H(z) = \arccosh \frac 1{1-z} = \sum_{l=0}^\infty H_l z^{l+1/2}
\]
so that, if we denote $x = 2 - 2 \epsilon$,
\[
\gamma_n(x) = \arccosh \frac{x}{2a_n} = H\left( \frac{\epsilon_n - \epsilon}{1-\epsilon} \right).
\]
Analogously to the analysis in the proof of Theorem~\ref{thm3},
\[
g(x) = \sum_{n=N_0}^{N(x)} H\left(\frac{\epsilon_n-\epsilon}{1-\epsilon} \right) = \sum_{l=0}^{L-1} H_l S_l(x) + O(\log(2-x))
\]
where $L = \lceil \frac 1{\tau} - \frac 12 \rceil$ and
\[
S_l(x) = \sum_{n=N_0}^{N(x)} \left(\frac{\epsilon_n-\epsilon}{1-\epsilon} \right)^{l+1/2} = (1-\epsilon)^{-l-1/2} \sum_{n=N_0}^{N(x)} (\epsilon_n-\epsilon)^{l+1/2}
\]
This is the formula that makes our choice of $H$ so useful: the factor independent of $n$ will not present further difficulty, and what remains is a power sum of $\epsilon_n - \epsilon$.

We first consider power sums of $\epsilon_n - \epsilon_N$, where $N = N(x)$, and use Lemma~\ref{lemma36} (with case $J=0$) to estimate the asymptotics of these sums as
\begin{align*}
\sum_{n=N_0}^{N(x)} (\epsilon_n-\epsilon_N)^{l+1/2} & = C^{l+1/2} \sum_{n=N_0}^{N} (n^{-\tau} - N^{-\tau})^{l+1/2} \\
& = C^{l+1/2} N^{1-\tau(l+1/2)} \int_0^1 (x^{-\tau} - 1)^{l+1/2} dx + O(1) \\
& = C^{l+1/2} N^{1-\tau(l+1/2)} \frac{\Gamma\left( \frac 1\tau - l - \frac 12\right) \Gamma\left( l + \frac 32\right)}{\Gamma\left(\frac 1\tau\right)} + O(1)
\end{align*}
and since the sum is monotonic in $x$ and $(N+1)^t - N^t = O(1)$ for $t \le 1$, we obtain
\[
S_l(x) = \frac{C^{l+1/2}}{(1-\epsilon)^{l+1/2}} \frac{\Gamma\left( \frac 1\tau - l - \frac 12\right) \Gamma\left( l + \frac 32\right)}{\Gamma\left(\frac 1\tau\right)} N^{1-\tau(l+1/2)}  + O(\log \epsilon)
\]
Since $N^t = C^{t/\tau} \epsilon_N^{t/\tau}$, using again $(N+1)^t - N^t = O(1)$ gives
\[
N^t = C^{t/\tau} \epsilon^{-t/\tau} + O(1), \qquad t \in (0,1].
\]
Putting this into $S_l(x)$ and combining $S_l(x)$ into $g(x)$ gives
\begin{align*}
g(x) & = \sum_{l=0}^{L-1} H_l C^{\frac 1\tau} \frac{\Gamma\left( \frac 1\tau - l - \frac 12\right) \Gamma\left( l + \frac 32\right)}{\Gamma\left(\frac 1\tau\right)}  \frac{  \epsilon^{-\frac 1\tau + l + \frac 12}  }{(1-\epsilon)^{l+1/2}} + O(\log \epsilon)
\end{align*}
It remains to simplify the resulting expression. Expanding $(1-\epsilon)^{-l-1/2}$ and using $\Gamma(l+3/2) = (l+1/2) \Gamma(l+1/2)$ gives
\[
g(x) = \sum_{l=0}^{L-1} \sum_{k=0}^{L-1-l} H_l C^{\frac 1\tau} \frac{\Gamma\left( \frac 1\tau - l - \frac 12\right) \left( l + \frac 12\right)}{\Gamma\left(\frac 1\tau\right)} \frac{\Gamma\left(l+k+\frac 12\right)}{k! } \epsilon^{-\frac 1\tau + l + \frac 12+k} + O(\log \epsilon)
\]
so
\[
\log f(x) = - \sum_{n=0}^{L-1} R_n \epsilon^{-\frac 1\tau + \frac 12 + n} + O(\log \epsilon)
\]
where
\[
R_n = 2 C^{\frac 1\tau} \frac{\Gamma\left(n+\frac 12\right)}{\Gamma\left(\frac 1\tau\right) } \sum_{l=0}^n H_l \frac{\Gamma\left( \frac 1\tau - l - \frac 12\right) \left( l + \frac 12\right)}{(n-l)! }
\]
Since
\[
H'(z) =  \frac 1{(1-z)\sqrt{2z - z^2}}
\]
we can compare coefficients to see that
\[
\left(l+\frac 12\right) H_l  = \frac 1{\sqrt 2} \sum_{m=0}^l  \binom{-1/2}m \left( - \frac 12\right)^m = \frac 1{\sqrt 2} \sum_{m=0}^l   \frac{\Gamma\left( m+ \frac 12\right)}{2^{m} m!}
\]
so
\[
R_n =  \sqrt 2 C^{\frac 1\tau} \frac{\Gamma\left(n+\frac 12\right)}{\Gamma\left(\frac 1\tau\right) } \sum_{l=0}^n  \sum_{m=0}^l   \frac{\Gamma\left( m+ \frac 12\right)}{2^{m} m!}  \frac{\Gamma\left( \frac 1\tau - l - \frac 12\right) }{(n-l)! }
\]
Since $\epsilon = (2-x)/2$, \eqref{asymp1d} follows with $S_n = R_n 2^{\frac 1\tau - \frac 12 - n}$.
\end{proof}

\section{Polynomially decaying Verblunsky coefficients}\label{sec4}

If $\gamma$ is a probability measure supported on $[-2,2]$, the Szeg\H o mapping  \cite[Section 13.1]{OPUC2} relates $\gamma$ to a probability measure $\nu$ supported on $\partial \mathbb{D}$ which is symmetric with respect to complex conjugation ($\theta \mapsto -\theta$) and such that for $g: [-2,2] \to \mathbb{R}$,
\[
\int_{[0,2\pi]} g(2 \cos \theta) d\nu(\theta) = \int_{[-2,2]} g(x) d\gamma(x).
\]
We note that $\gamma(\{2\}) > 0$ if and only if $\nu(\{1\}) > 0$. Moreover, if $d\gamma(x) = f(x) dx + d\gamma_\s$ and $d\nu(\theta) = v(\theta) \frac{d\theta}{2\pi} + d\nu_\s$ are Lebesgue decompositions of $\gamma$ and $\nu$, then
\begin{equation}\label{4.1}
v(\theta) = 2 \pi \lvert \sin\theta\rvert f(2\cos \theta).
\end{equation}
Formulas of Geronimus express Jacobi parameters $\{a_n, b_n\}_{n=1}^\infty$ of $\gamma$ in terms of Verblunsky coefficients $\{\beta_n\}_{n=0}^\infty$ of $\nu$:
\begin{align*}
a_{n+1}^2 & = (1-\beta_{2n-1})(1-\beta_{2n}^2)(1+\beta_{2n+1}) \\
b_{n+1} & = (1-\beta_{2n-1}) \beta_{2n} - (1+\beta_{2n-1}) \beta_{2n-2}
\end{align*}
with the convention $\beta_{-1} = -1$.

In particular, if $\nu$ is the sieved measure \cite[Section 1.6]{OPUC1} obtained from  $\mu$,
\[
d\nu (\theta) = \frac 12 d\mu(2 \theta),
\]
 then $v(\theta) = w(2 \theta)$ and 
\[
\beta_n = \begin{cases} \alpha_{(n-1)/2} & n\equiv 1 \pmod{2} \\
0 &  n\equiv 0 \pmod{2}
\end{cases}
\]
Combining, we see that Jacobi parameters of $\gamma$ are given in terms of Verblunsky coefficients of $\mu$ by \eqref{Szegosieving}, that $\mu(\{1\}) > 0$ if and only if $\gamma(\{2\}) > 0$, and that
\[
w(\theta) =  2 \pi \left\lvert \sin\frac \theta 2\right\rvert f\left(2\cos \frac \theta 2\right)
\]
so
\begin{equation}\label{Szegosieving2}
\log w(\theta) = \log  f\left(2\cos \frac \theta 2\right) + O(\log \lvert\theta\rvert), \qquad \theta \to 0.
\end{equation}

The first step is to verify the absence of the mass point.

\begin{lemma}
If $\alpha_n \in \mathbb{R}$ for all $n$ and $\alpha_n <0$ for all $n\ge N_0$, then $\mu(\{1\})=0$.
\end{lemma}

\begin{proof}
By induction, using \eqref{Szegorecursion}, $\varphi_n(1)$ is real for all $n$ and
\[
\varphi_n(1) = \prod_{j=0}^{n-1} \sqrt{\frac{1-\alpha_j}{1+\alpha_j}}.
\]
Since $\alpha_n < 0$ for all $n\ge N_0$, it follows that $\varphi_n(1)$ is a positive increasing sequence. 
By \cite[Thm.~2.7.3]{OPUC1}, $(\varphi_n(1))_{n=0}^\infty \not\in \ell^2$ implies $\mu(\{1\})=0$.
\end{proof}

\begin{proof}[Proof of Theorem~\ref{thm5}]
By \eqref{Szegosieving}, $a_n$ is of the form \eqref{anansatz} with leading terms
\[
a_n = 1 - \frac 12 D_1^2 n^{-2\tau_1} + \dots
\]
and $b_n \equiv 0$, so Theorem~\ref{thm2} applies and gives \eqref{asymptotic1000} and $C_1 = D_1^2$, $\beta = 2 \tau_1$. The case $\tau_1 \ge 1$ follows immediately from Theorem~\ref{thm2}(a) and \eqref{Szegosieving2}. Assume from now on that $\tau_1 \in (0,1)$. We wish to plug in
\begin{equation}\label{owgijsd}
x = 2 \cos \frac \theta 2
\end{equation}
and use \eqref{Szegosieving2}, so we note that
\[
2 - x = 2 - 2 \cos \frac \theta 2 =\frac{\theta^2}4 s(\theta)
\]
where
\[
s(\theta) = \frac{ 2 - 2 \cos \frac \theta 2 }{ \frac{\theta^2}4 } = 1 + \sum_{k=2}^\infty \frac{(-1)^{k-1}}{2^{2k-3} (2k)!}\theta^{2k-2}
\]
is an entire function with $s(0)=1$, so
\begin{equation}\label{asymp45}
(2 - x)^{-\kappa} =  \frac{\theta^{-2\kappa}}{2^{-2\kappa}} s(\theta)^{-\kappa}
\end{equation}
where $s(\theta)^{-\kappa}$ has a power series representation around $\theta=0$ with positive radius of convergence for any $\kappa$. Using \eqref{Szegosieving2} and \eqref{asymptotic1000} and keeping all (finitely many) terms with negative powers of $\theta$ gives \eqref{asymptotic10}. The leading term comes from the leading terms of \eqref{asymptotic1000}, \eqref{asymp45}, so it is
\[
(D_1^2)^{\frac 1{2\tau_1}} \frac{\Gamma\left(\frac 1{2\tau_1} - \frac 12\right) \sqrt \pi }{ \Gamma\left(\frac 1{2\tau_1}  \right) }  \left( \frac{\theta^2}4 \right)^{-\frac 1{2\tau_1} + \frac 12}
\]
which implies \eqref{asymp13}.
\end{proof}

\begin{proof}[Proof of Theorem~\ref{thm6}]
By \eqref{Szegosieving},
\begin{align*}
a_n^2  = 1 - D^2 n^{-2\tau} + D \tau n^{-1-\tau} + D^2 \tau n^{-1-2\tau} + o(n^{-1-2\tau})
\end{align*}
and $b_n \equiv 0$. Following verbatim the approach from Section~\ref{sec3} would be impractical: as derived in \eqref{anszegomap}, the formula for $a_n$ involves all powers $n^{-2k\tau}$ with $k = 1, 2, \dots, \lfloor \frac 1{2\tau} +1 \rfloor$, which would make all the following formulas very complicated with arbitrary $\tau$. Instead, we make a modification suited to the form of our Jacobi parameters: we write
\[
a_n = \sqrt{1-\epsilon_n}, \qquad \epsilon_n = D^2 n^{-2\tau} - D \tau n^{-1-\tau} - D^2 \tau n^{-1-2\tau} + o(n^{-1-2\tau}),
\]
and instead of the function $F(z) = \arccosh\left(1+\frac z2\right)$ consider 
\[
G(z) = \arccosh \frac 1{\sqrt{1-z}}.
\]
Let us also define $\epsilon$ by $x = 2 \sqrt{1-\epsilon}$, noting that combining that definition with  \eqref{owgijsd} implies
\[
\epsilon = \sin^2 \frac{\theta}2.
\]
The choice of the function $G(z)$ allows us to write, denoting $a_n^2 = 1 - \epsilon_n$,
\[
\gamma_n(x) = \arccosh\frac{2\sqrt{1-\epsilon}}{2\sqrt{1-\epsilon_n}} = G\left(\frac{\epsilon_n-\epsilon}{1-\epsilon} \right).
\]
It is elementary to verify that $G(z)$ has a simpler representation as
\[
G(z) = \frac 12 \ln \frac{1+\sqrt z}{1-\sqrt z} = \sum_{l=0}^\infty \frac 1{2l+1} z^{l + 1/2}
\]
so, analogously to the analysis in Section~\ref{sec3},
\[
g(x) = \sum_{n=N_0}^{N(x)} G\left(\frac{\epsilon_n-\epsilon}{1-\epsilon} \right) = \sum_{l=0}^{L-1} \frac 1{2l+1} S_l(x)
\]
where $L = \lceil \frac 1{2\tau} - \frac 12 \rceil$ and
\[
S_l(x) = \sum_{n=N_0}^{N(x)} \left(\frac{\epsilon_n-\epsilon}{1-\epsilon} \right)^{l+1/2} = (1-\epsilon)^{-l-1/2} \sum_{n=N_0}^{N(x)} (\epsilon_n-\epsilon)^{l+1/2}
\]
Once again, this is the formula that makes our choice of $G$ useful: the factor independent of $n$ will not present further difficulty, while the remainder is a power sum of $\epsilon_n - \epsilon$, with $\epsilon_n$ which have comparatively few terms in their asymptotic behavior to order $o(n^{-1-\beta})$.

As in the proof of Theorem~\ref{thm3}, the term $o(n^{-1-\tau})$ from $\epsilon_n$ can be removed by an adaptation of Lemma~\ref{lemma32b}, so from now on let us assume that
\[
\epsilon_n = D^2 n^{-2\tau} - D \tau n^{-1-\tau} - D^2 \tau n^{-1-2\tau}.
\]
It follows from Lemma~\ref{lemma32} that there exists $N_1$ such that for $N\ge N_1$, for any $n$ with $N_1 \le n \le N-1$,
\[
\lvert (\epsilon_n - \epsilon_N) - D^2 (n^{-2\tau} - N^{-2\tau}) \rvert < 2 D \tau (n^{-1-\tau} - N^{-1-\tau}) < \frac 12 (n^{-2\tau} - N^{-2\tau})
\]
and we can Taylor expand for $N_1 \le n \le N-1$
\begin{align*}
& (\epsilon_n - \epsilon_N)^{l+1/2}  \\
& = \left(D^2(n^{-2\tau} - N^{-2\tau})\right)^{l+1/2} \left( 1+ \frac{2D\tau(n^{-1-\tau}- N^{-1-\tau})+D^2(n^{-1-2\tau}- N^{-1-2\tau})}{D^2(n^{-2\tau} - N^{-2\tau})} \right)^{l+1/2} \\
& = D^{2l+1} (n^{-2\tau} - N^{-2\tau})^{l+1/2} \left( 1 + O\left(\frac{n^{-1-\tau} - N^{-1-\tau}}{n^{-2\tau} - N^{-2\tau}} \right) \right)
\end{align*}
Considering the sum
\begin{equation}\label{afhgso}
\sum_{n=N_1}^{N-1} (n^{-2\tau} - N^{-2\tau})^{l+1/2} \left(\frac{n^{-1-\tau} - N^{-1-\tau}}{n^{-2\tau} - N^{-2\tau}} \right)
\end{equation}
in the notation of Lemma~\ref{lemma36}, we have
\[
\lambda = 2\tau\left( l+ \frac 12\right) + (1+\tau) - 2\tau = 2 l \tau + 1 \ge 1
\]
so by Lemma~\ref{lemma36}, the sum \eqref{afhgso} is of order $O(\log N)$. Therefore,
\[
\sum_{n=N_0}^{N} (\epsilon_n - \epsilon_N )^{l+1/2}  = D^{2l+1} \sum_{n=N_0}^{N} (n^{-2\tau} - N^{-2\tau})^{l+1/2}  + O(\log N)
\]
Applying Lemma~\ref{lemma36} and using $x^{2\tau}=u$ to reduce the integral to the Beta function,
\begin{align*}
\sum_{n=N_0}^{N} (\epsilon_n - \epsilon_N )^{l+1/2}  & = D^{2l+1} N^{1-(2l+1)\tau} \int_0^1 (x^{-2\tau}-1)^{l+1/2} dx +O(\log N) \\
& = D^{2l+1} N^{1-(2l+1)\tau} \frac 1{2\tau} \int_0^1 u^{\frac 1{2\tau} - l - \frac 32} (1-u)^{l+1/2} du +O(\log N) \\
& = D^{2l+1} N^{1-(2l+1)\tau} \frac {\Gamma\left( \frac 1{2\tau} - l - \frac 12 \right) \Gamma\left( l + \frac 32 \right)}{2\tau \Gamma\left( \frac 1{2\tau} + 1 \right) }  +O(\log N).
\end{align*}
Since
\[
\epsilon_N = D^2 N^{-2\tau} (1+O(N^{-1+\tau})),
\]
for $s > 0$,
\[
\epsilon_N^{-s} = D^{-2s} N^{2s\tau} (1 + O(N^{-1+\tau}))
\]
so if $2 s \tau \le 1 - \tau$ we obtain
\[
\epsilon_N^{-s} = D^{-2s} N^{2s\tau} + O(1)
\]
Solving for $N$ and using $t = 2s\tau$, we see that if $0 < t \le 1 - \tau$, then
\[
N^t = D^{\frac t \tau} \epsilon_N^{- \frac t{2\tau}} + O(1)
\]
and using $(N+1)^{t} - N^{t} = O(1)$ and $\epsilon_{N+1} < \epsilon \le \epsilon_N$ we conclude that also
\[
N^t = D^{\frac t \tau} \epsilon^{- \frac t{2\tau}} + O(1).
\]
Putting this into $S_l(x)$ and combining $S_l(x)$ into $g(x)$ gives
\begin{align*}
g(x) & = \sum_{l=0}^{L-1} \frac {D^{2l+1}}{2l+1}  \frac {\Gamma\left( \frac 1{2\tau} - l - \frac 12 \right) \Gamma\left( l + \frac 32 \right)}{2\tau \Gamma\left( \frac 1{2\tau} + 1 \right) }  \frac{N^{1-(2l+1)\tau}}{ (1-\epsilon)^{l+\frac 12}} + O(\log \epsilon) \\
& = \sum_{l=0}^{L-1} \frac {D^{\frac 1\tau}}{2}  \frac {\Gamma\left( \frac 1{2\tau} - l - \frac 12 \right) \Gamma\left( l + \frac 12 \right)}{ \Gamma\left( \frac 1{2\tau} \right) }  \frac{\epsilon^{-\frac 1{2\tau} + \frac 12 + l}}{ (1-\epsilon)^{l+\frac 12}} + O(\log \epsilon)
\end{align*}
Writing a power series expansion of $(1-\epsilon)^{-l-1/2}$ as
\[
(1-\epsilon)^{-l-\frac 12} = \sum_{m=0}^\infty \binom{-l-\frac 12}m (-\epsilon)^m = \sum_{m=0}^\infty \frac{\Gamma\left(l+m+\frac 12\right)}{m! \Gamma\left( l+ \frac 12\right)} \epsilon^m
\]
we obtain
\[
g(x) = \sum_{l=0}^{L-1} \sum_{m=0}^{L-1-l} \frac {D^{\frac 1\tau}}{2}  \frac {\Gamma\left( \frac 1{2\tau} - l - \frac 12 \right) \Gamma\left( l + m + \frac 12 \right)}{ m! \Gamma\left( \frac 1{2\tau} \right) } \epsilon^{-\frac 1{2\tau} + \frac 12 + l+m} + O(\log \epsilon)
\]
so the result follows by grouping terms by $n=l+m$ since $\epsilon = \sin^2 \frac \theta 2$.
\end{proof}

\bibliographystyle{amsplain}
%\bibliography{biblio}{}

\providecommand{\bysame}{\leavevmode\hbox to3em{\hrulefill}\thinspace}
\providecommand{\MR}{\relax\ifhmode\unskip\space\fi MR }
% \MRhref is called by the amsart/book/proc definition of \MR.
\providecommand{\MRhref}[2]{%
  \href{http://www.ams.org/mathscinet-getitem?mr=#1}{#2}
}
\providecommand{\href}[2]{#2}

\end{document}